\documentclass{article}

\usepackage[english]{babel}

\usepackage[letterpaper,top=2cm,bottom=2cm,left=3cm,right=3cm,marginparwidth=1.75cm]{geometry}

\usepackage{amsmath}
\usepackage{graphicx}
\usepackage[colorlinks=true, allcolors=blue]{hyperref}
\usepackage{float}
\usepackage{multirow}

\usepackage[english]{babel}
\usepackage{amsmath,amsfonts,amssymb,amsthm,mathtools}
\usepackage{bm,bbm}
\usepackage{mathrsfs}
\usepackage{empheq}
\usepackage{enumerate}

\newtheorem{theorem}{Theorem}
\newtheorem{proposition}{Proposition}
\newtheorem{lemma}{Lemma}
\newtheorem{corollary}{Corollary}

\newcommand{\Z}{{\mathbb{Z}}}
\newcommand{\N}{\mathbb{N}}
\newcommand{\R}{\mathbb{R}}
\newcommand{\C}{\mathbb{C}}

\title{Deconvolution of spherical data corrupted with unknown noise}

\author{Jérémie Capitao-Miniconi, Elisabeth Gassiat \\
 Universit{\'e} Paris-Saclay, CNRS, Laboratoire de Math{\'e}matiques d'Orsay,\\  91405 Orsay, France.}

\date{}

\begin{document}
\maketitle

\begin{abstract}
We consider the deconvolution problem for densities supported on a $(d-1)$-dimensional sphere with unknown center and unknown radius, in the situation where the distribution of the noise is unknown and without any other observations. We propose estimators of the radius, of the center, and of the density of the signal on the sphere that are proved consistent without further information. The estimator of the radius is proved to have almost parametric convergence rate for any dimension $d$. When $d=2$, the estimator of the density is proved to achieve the same rate of convergence over Sobolev regularity classes of densities as when the noise distribution is known.
\end{abstract}

\section{Introduction}

In this paper, we study the deconvolution problem of random data on a sphere corrupted by independent additive noise. The observations are
\begin{equation}
\label{eq:model}
Y_{i}=X_{i}+\varepsilon_{i}, \;i=1,\ldots,n
\end{equation}
where $(X_i)_{i\geq 1}$ (the signal)  is a sequence of independent identically distributed (i.i.d.) random variables taking values on a $(d-1)$-dimensional sphere (for $d \geq 2$) with unknown center $C^{\star}$ and unknown radius $R^{\star}$, $(\varepsilon_{i})_{i\geq 1}$ (the  noise) is a sequence of i.i.d. random variables independent of the signal and with totally unknown distribution. The distribution of the signal is also unknown, it is only known that it is spherically supported. To solve the deconvolution problem and estimate the structural parameters $C^{\star}$ and $R^{\star}$, the only assumption we shall put on the noise is that its $d$ coordinates are independently distributed.\\ 

The statistical estimation of the center and of the radius of the sphere is of interest in various applications
such as object tracking, robotics, pattern recognition, see for instance \cite{EF18}, \cite{EF20}, \cite{Pan11}, among others, see also \cite{BC86} and references therein. Several methods have been proposed based on least squares, maximum likelihood, see \cite{Hypersphere2021}
for a recent likelihood based algorithm, most of them modeling the noise distribution with a Gaussian distribution.\\  

The deconvolution problem of the distribution of the signal when the radius and the center are known is studied for circular signals (that is when $d=2$) in \cite{Go2002}. The author proves that the minimax  rate of convergence of the estimator over  a wide collection of smoothness classes of the density of the signal on the circle  does not depend on the (known) noise distribution, for a variety of different noise distributions, contrasting with the situation where the signal has a density with respect to Lebesgue over the whole space.\\

Recently, it has been proved in \cite{LCGL2021} that deconvolution  with unknown noise distribution is possible for multivariate signals,  as soon as the signal can be decomposed in two components that satisfy a mild dependence assumption, that its distribution has light enough tails, and without any assumption on the noise distribution except that its two corresponding components are independently distributed. The authors of \cite{LCGL2021} then consider the situation where the probability of $X_1$ has a density with respect to Lebesgue measure, and they prove that not knowing the noise distribution does not deteriorate the estimation rate of the density on Sobolev regularity classes for compactly supported signals.\\

Here, the probability distribution of the signal is singular with respect to Lebesgue measure on $\R^d$ and their convergence results do not apply. However, we prove that the general conditions they propose under which deconvolution with unknown noise is possible is satisfied for spherical signals, this is our first main identifiability result Theorem \ref{theo:ident}.
The main contribution of our work is then to exhibit estimators that achieve remarkable properties:
\begin{itemize}
\item
We propose estimators of the radius, the center, and the distribution of the signal, which are proved consistent whatever the noise distribution, see Proposition \ref{prop:consi}.
\item
Under the mild assumption that the noise has finite variance, we prove that the radius of the sphere can be estimated at almost parametric rate with totally unknown noise distribution, see Theorem \ref{theo:radius}.
\item
When $d=2$, that is for circular signals, we prove that the center can be estimated at almost parametric rate and that the density of the signal distribution on the circle can be estimated at the same rate as when the distribution of the noise  is known on some Sobolev regularity classes, with a rate which is minimax as proved in \cite{Go2002}, see Theorem \ref{theo:density} and Theorem \ref{theo:center}.
\end{itemize}

In Section \ref{sec:ident}, we first recall general results of \cite{LCGL2021} and we prove in Proposition \ref{prop:phihat} a strengthened version of the local $L^2$-consistency of the general estimator of the characteristic function of the signal that will be a basic stone for all our convergence rates theorems. We then state our identifiability theorem, give the definition of the estimators and prove their consistency. Section \ref{sec:rates} studies the rates of convergence of our estimators, and in section \ref{sec:semipara} we study the situation where radius and center of the sphere together with the noise distribution are unknown, though the distribution of the angles of the random signal is known. Simulations illustrating our findings are given in Section \ref{sec:simu}. We discuss possible further work and related questions in Section \ref{sec:discu}.   Proofs of  propositions and lemmas are detailed in Section \ref{sec:proofs}.

\section{Identifiability and estimation method}
\label{sec:ident}

In this section, we prove that model \eqref{eq:model} is identifiable with no more assumptions.  We then explain the estimation method and define the estimators which will be studied in Section \ref{sec:rates}.

\subsection{Preliminaries: deconvolution with unknown noise}
\label{subsec:prelim}
We first recall general results in \cite{LCGL2021}. Then, we prove a proposition
which will be used to obtain the nearly parametric rate of our estimators of the radius and the center. 
 In \cite{LCGL2021}, the authors consider the situation where the observations $Y_i \in \R^d$ come from the model
 $$
 \begin{pmatrix} Y^{(1)}\\Y^{(2)}\end{pmatrix} = \begin{pmatrix} X^{(1)}\\X^{(2)}\end{pmatrix} + \begin{pmatrix} \varepsilon^{(1)}\\ \varepsilon^{(2)}\end{pmatrix}
 $$
in which $Y^{(1)}\in\R^{d_1}$,  $d_1 \geq 1$, and $Y^{(2)}\in\R^{d_2}$,  $d_2 \geq 1$, with $d_1+d_2 = d$, and where $\varepsilon^{(1)}$ is independent of $\varepsilon^{(2)}$. They prove identifiability under  very mild assumptions on the signal distribution. The first one is about the tail of its distribution.
\begin{itemize}
\item[A($\rho$)]
There exists $\rho <2$, $a>0$ and $b>0$ such that for all $\lambda\in\R^d$,
$E\left[\exp \left(\lambda^\top X\right)\right]
\leq a \exp \left( b \|\lambda\|_{2}^\rho\right)$.
\end{itemize}
Here, $X=\begin{pmatrix} X^{(1)}\\X^{(2)}\end{pmatrix}$ and  $\|\cdot\|_{2}$ is the Euclidian norm.\\
Under A($\rho$),
the characteristic function of the signal can be extended into a multivariate analytic function denoted by
\begin{eqnarray*}
\Phi: \C^{d_1}\times \C^{d_2} &\longrightarrow& \C\\
(z_1,z_2)&\longmapsto& E\left[ \exp \left(iz_1^\top X^{(1)} + i z_2^\top X^{(2)}\right)\right].
\end{eqnarray*}
The second assumption is a mild dependence assumption (see the discussion after Theorem 2.1 in \cite{LCGL2021}).
\begin{itemize}
\item[A(dep)]
For any $z_{0}\in \C^{d_1}$, 
$z \mapsto \Phi (z_{0},z)$
is not the null function and for any $z_{0}\in \C^{d_2}$, 
$z \mapsto \Phi (z,z_{0})$
is not the null function.
\end{itemize}
Obviously, if no centering constraint is put on the signal or on the noise, it is possible to translate the signal by a fixed vector $m\in\R^d$ and the noise by  $-m$ without changing the observation. The model can thus  be identifiable only up to translation. 

\begin{theorem}[from \cite{LCGL2021}]
\label{theo:prelim}
If the distribution of the signal satisfies A($\rho$) and
A(dep), then the distribution of the signal and the distribution of the noise can be recovered from the distribution of the observation, up to  translation.
\end{theorem}

The first step in the estimation procedure is the estimation of the characteristic function of the signal by a method inspired by the proof of the identifiability theorem. 
For any $S>0$, let $\Upsilon_{\rho,S}$ be the subset of multivariate analytic functions from $\C^{d}$ to $\C$ defined as follows.
$$
\Upsilon_{\rho,S} = \left\{
\phi \text{ analytic } \text{s.t. } \forall z\in\R^{d},\overline{\phi(z)}=\phi(-z), \phi(0) = 1 \text{ and } \forall j \in \N^d \setminus \{0\}, \left| \frac{\partial^{j} \phi(0)}{\prod_{a=1}^d j_a!}\right| \leq \frac{S^{\|j\|_1}}{\|j\|_1^{ \|j\|_1 /\rho}} \right\}
$$
where $\|j\|_1 =\sum_{a=1}^{d} i_{a}$. For all $\Phi$ satisfying A($\rho$), there exists $S$ such that $\Phi \in
\Upsilon_{\rho,S}$. Let $\Phi_{\varepsilon^{(p)}}$ be the characteristic function of $\varepsilon^{(p)}$, $p=1,2$, and define for all $\phi\in \Upsilon_{\rho,S}$ and any $\nu >0$,
\begin{equation*}
M(\phi; \nu \vert \Phi) = \int_{B_{\nu}^{d_1}\times B_{\nu}^{d_2}} | \phi(t_1,t_2) \Phi(t_1,0) \Phi(0,t_2) - \Phi(t_1,t_2) \phi(t_1,0) \phi(0,t_2) |^2 
|\Phi_{\varepsilon^{(1)}}(t_1)\Phi_{\varepsilon^{(2)}}(t_2)  |^2 d t_1 d t_2,
\end{equation*}
where $B_{\nu}=[-\nu,\nu]$. It is proved in \cite{LCGL2021} that if $\phi\in \Upsilon_{\rho,S}$ satisfies A(dep), 
$M(\phi; \nu \vert \Phi)=0$ for a fixed $\nu$ if and only if $\phi=\Phi$ (up to translation). The estimator of the characteristic function of the signal can then be defined as a minimizer of the empirical estimator of $M$.

Fix some $\nu_{\text{est}}>0$.
Let $\cal H$ be a subset of functions from $\R^d$ to $\C^d$ such that all elements of $\cal H$ satisfy A(dep) and which is closed in $L^{2}(B_{\nu_{\text{est}}}^d)$. Define $\widehat \phi_{n}$ as a (up to $1/n$) measurable minimizer of the functional $M_n$ over $\Upsilon_{\rho,S}\cap {\cal H}$, where $M_n$ is defined as
\begin{equation*}
M_{n}(\phi) = \int_{B_{\nu_{\text{est}}}^{d_1}\times B_{\nu_{\text{est}}}^{d_2}} | \phi(t_1,t_2) \tilde\phi_n(t_1,0) \tilde\phi_n(0,t_2) - \tilde \phi_n(t_1,t_2) \phi(t_1,0) \phi(0,t_2) |^2 d t_1 d t_2 ,
\end{equation*}
where for all $(t_1,t_2)\in\C^{d_1}\times \C^{d_2}$,
\begin{equation*}
\tilde \phi_n(t_1,t_2) = \frac{1}{n}\sum_{\ell=1}^{n} \exp\left\{it_1^\top Y_{\ell}^{(1)} + it_2^\top Y_{\ell}^{(2)}\right\}.
\end{equation*}

It appears that, for any $\nu >0$, $\widehat \phi_{n}$ is a consistent estimator of $\Phi$ in
$L^{2}([-\nu,\nu]^d)$ at almost parametric rate. The constants will depend on the signal through $\rho$ and $S$, and on the noise through its second moment and the following quantity:
\begin{equation}
\label{eq:cnu}
c_{\nu}=\inf\{|\Phi_{\varepsilon^{(1)}}(t)|,\;t\in B_{\nu}^{d_1}\} \wedge 
\inf\{|\Phi_{\varepsilon^{(2)}}(t)|,\;t\in B_{\nu}^{d_2}\}.
\end{equation}
For any noise distribution, for small enough $\nu$, $c_{\nu}$ is a positive real number.
We prove the following.
\begin{proposition}
\label{prop:phihat}
Assume $\Phi\in\Upsilon_{\rho,S}\cap {\cal H}$ and  $\varepsilon_1$ has finite variance.
Fix some $\nu\in[(d+4/3) e/S,\nu_\text{est}]$ such that $c_{\nu}>0$.
For all $\delta >0$, 
there exist positive constants $c_{1}, c_{2}$, $c_{3}$ which depend on $\delta$, $\nu_\text{est}$, $\nu$, $c_{\nu}$,$\rho$,  $S$, $\cal H$, $d$ and $E(\|Y_{1}\|^2)$ such that for all $x \geq 1$ and $n\geq (1\vee xc_{1})/c_{2}$,
with probability at least $1 - e^{-x}$,
$$
\int_{B_{\nu}^{d}} |\widehat \phi_{n} (t)- \Phi(t)|^{2} dt \leq c_{3}\left(\frac{x}{n^{1-\delta}} \vee \frac{x^2}{n^{2-2\delta}} \right).
$$
\end{proposition}
Proposition \ref{prop:phihat} improves on Proposition A.3 in \cite{LCGL2021} and is proved in Section \ref{subsec:phihat}.

\subsection{Identifiability theorem}
\label{subsec:ident}

For any $Z\in \R^{d}$, denote $Z^{(1)}, \ldots , Z^{(d)}$ its $d$ coordinates. We shall parametrize a vector on a sphere through angles. For any $u\in[0,1]^{d-1}$, define $S(u)$ on the unit $d$-dimensional sphere as 
$$
S(u)=\begin{pmatrix} \cos(2 \pi u^{(1)}) \\ \sin(2 \pi u^{(1)}) \cos(\pi u^{(2)}) \\ \sin(2 \pi u^{(1)}) \sin( \pi u^{(2)}) \cos( \pi u^{(3)}) \\ \vdots \\ \sin(2 \pi u^{(1)}) \sin(\pi u^{(2)}) \cdots \sin(\pi u^{(d-2)}) \cos(\pi u^{(d-1)}) \\ \sin(2 \pi u^{(1)}) \sin(\pi u^{(2)}) \cdots \sin(\pi u^{(d-2)}) \sin(\pi u^{(d-1)})  \end{pmatrix}.
$$

Then for a sequence $(U_{i})_{i\geq 1}$ of i.i.d random vectors taking values in $[0,1]^{d-1}$, we have for all $i\geq 1$,
\begin{equation}
\label{eq:sphere}
X_{i}=C^{\star} + R^{\star}S(U_{i}),
\end{equation}
with $C^{\star}$ the center of the sphere and $R^{\star}$ its radius.

We shall also make the following assumptions.
\begin{itemize}
\item[(H1)]
The coordinates  $\varepsilon_{1}^{(1)},\ldots,\varepsilon_{1}^{(d)}$ of the  noise are independently distributed. We denote $\mathbb{Q}^{\star} =  \otimes_{j=1}^d \mathbb{Q}^{\star}_j$ the distribution of $\varepsilon_{1}$, with $\mathbb{Q}_j^{\star}$ the distribution of $\varepsilon_{1}^{(j)}$, $j=1,\ldots,d$.
\item[(H2)]
The distribution of $U_{1}$ has a density $f^{\star}$ with respect to Lebesgue measure on $[0,1]^{d-1}$. When $d>2$, we assume $f^{\star}$ positive on $(0,\gamma)^{d-1}$ for some $\gamma >0$.
\end{itemize}
We shall sometimes call $f^{\star}$ {\it exploration density of the angles} or {\it exploration density}.
For any $\mathbb{Q} =  \otimes_{j=1}^d \mathbb{Q}_j$, with $\mathbb{Q}_j$, $j=1,\ldots,d$, probability distributions on $\R$, any probability density $f$ on $[0,1]^{d-1}$, any $C\in \R^d$ and any $R\geq 0$, let $\mathbb{P}_{C,R,f,\mathbb{Q}}$ be the distribution of $Y_{1}$ when $X_{1}$ lies on the sphere with center $C$, radius $R$, and  $U_1$ has density $f$.

\begin{theorem}
\label{theo:ident}
Assume (H1) and (H2). For any $\mathbb{Q} =  \otimes_{j=1}^d \mathbb{Q}_j$, any probability density $f$ on $[0,1]^{d-1}$, any $C\in \R^d$ and any $R\geq 0$, $\mathbb{P}_{C,R,f,\mathbb{Q}}=\mathbb{P}_{C^{\star},R^{\star},f^{\star},\mathbb{Q}^{\star}}$ if and only if
$R=R^{\star}$, $f=f^{\star}$, and there exists $m\in\R^{d}$ such that $C=C^{\star}+m$ and $\mathbb{Q}(\cdot)=\mathbb{Q}^{\star}(\cdot +m)$.
If moreover $\mathbb{Q}$ and $\mathbb{Q}^{\star}$ have finite first moment and are centered distributions, then $m=0$, that is $C=C^{\star}$ and $\mathbb{Q}=\mathbb{Q}^{\star}$.
\end{theorem}

\noindent
{\bf Proof of Theorem \ref{theo:ident}.}\\
We shall apply Theorem \ref{theo:prelim}. For any probability density $f$ on $[0,1]^{d-1}$, $C\in \R^d$ and $R\geq 0$, A($\rho$) holds with $\rho =1$ and with the constants $a=1$ and $b=\|C\|_{2}+R$. To verify A(dep), since all coordinates of the noise are independently distributed, we first  choose a decomposition of the signal in two components. We define 
$\widetilde{X}^{(1)}=X_{1}^{(1)}$, and $\widetilde{X}^{(2)}=(X_{1}^{(2)},\ldots,X_{1}^{(d)})^T$, and
we prove in Section \ref{subsec:condi} the following Lemma from which A(dep) follows.
\begin{lemma}
\label{lem:condi}
Assume (H2). Then for all 
$z\in\C$, $E\left[\exp \left(iz \widetilde{X}^{(1)}\right) \vert \widetilde{X}^{(2)}\right]$ is not $P_{\widetilde{X}^{(2)}}$-a.s. the null random variable, and for all \
$z\in\C^{d-1}$, $E\left[\exp\left( iz^T  \widetilde{X}^{(2)}\right) \vert \widetilde{X}^{(1)}\right]$ is not $P_{\widetilde{X}^{(1)}}$-a.s. the null random variable. Here, $P_{\widetilde{X}^{(p)}}$ denotes the distribution of ${\widetilde{X}^{(p)}}$, $p=1,2$.
\end{lemma}
Then,   translation of the spherical signal does not change the radius of the sphere and the exploration density of the angles on the sphere, but only the centering of the sphere and correspondingly the distribution of the noise. Applying Theorem \ref{theo:prelim} leads then to the conclusion of Theorem \ref{theo:ident}.

The proof of Lemma \ref{lem:condi} proceeds by computing explicitly the conditional expectation, and then to give an argument why it can not be the null random variable. The argument for $d=2$ does not apply to $d>2$, in which case we use another argument needing the positivity of $f^{\star}$ near the origin. Since the choice of the positive first coordinate  to define the angles and the density is arbitrary, the proof still holds under the assumption that the density is positive near the point of the sphere at the intersection with one of the $2d$ axis directions.

\subsection{Estimation method and consistency}
\label{subsec:method}
We shall apply the method described in Section \ref{subsec:prelim} to estimate $R^{\star}$ and $f^{\star}$. For any positive real number $R$ and any probability density $f$ on $[0,1]^{d-1}$, define
$\Psi_{f,R}$ the characteristic function of the random variable with distribution on the centered sphere with radius $R$, and exploration density of the angles $f$, that is, for all $t\in \R^d$,
\begin{equation} 
\label{eq:psifR}
\Psi_{f,R}(t)=\int_{(0,1)^{d-1}}\exp \left\{i R t^{T}S(u) \right\}f(u)du.
\end{equation}
We shall consider functions $\Psi_{f,R}$ for any function $f$ on $(0,1)^{d-1}$
(not only probability densities) as defined by \eqref{eq:psifR}. Notice that  $\Psi_{f,R}$ can be extended to $\C^d$.
\\
Since all components of $\epsilon_1$ are independent, we have to make a choice of $d_1$ and $d_2$ for the definition of $M_n$ and $M$, thus, in the following, we assume to have $d_1 = 1$ and $d_2 = d-1$, as in the proof of Theorem \ref{theo:ident}. For any $\nu > 0$, define
\begin{equation*}
{M}(f,R) = \int_{B_{\nu} \times B_{\nu}^{d-1} } | \Psi_{f,R}(t_1,t_2) \Psi_{f^{\star},R^{\star}}(t_1,0) \Psi_{f^{\star},R^{\star}}(0,t_2) - \Psi_{f^{\star},R^{\star}}(t_1,t_2) \Psi_{f,R}(t_1,0) \Psi_{f,R}(0,t_2) |^2 |\Phi_{\epsilon}(t_1,t_2)|^2 d t_1 d t_2.
\end{equation*}
The parameter $\nu$ does not appear in the notation of $M$ and can be chosen as needed. \\
Fix some $\nu_{\text{est}}>0$, and define
\begin{equation*}
{M_n}(f,R) = \int_{B_{\nu_{\text{est}}}\times B_{\nu_{\text{est}}}^{d-1}} | \Psi_{f,R}(t_1,t_2) \tilde\psi_n(t_1,0) \tilde\psi_n(0,t_2) - \tilde \psi_n(t_1,t_2) \Psi_{f,R}(t_1,0) \Psi_{f,R}(0,t_2) |^2 d t_1 d t_2 ,
\end{equation*}
with
\begin{equation*}
\tilde \psi_n(t_1,t_2) = \frac{1}{n}\sum_{\ell=1}^{n} \exp\left\{it_1  \tilde{Y}_{\ell}^{(1)} + it_2^\top \tilde{Y}_{\ell}^{(2)}\right\}.
\end{equation*}
where for all $\ell$,
$\tilde{Y}_{\ell}^{(1)}=Y_{\ell}^{(1)}$, and $\tilde{Y}_{\ell}^{(2)}=(Y_{\ell}^{(2)},\ldots,Y_{\ell}^{(d)})^T$.\\
We need to fix the compact subset on which we minimize $M_{n}$. We choose ${\cal F}$ a compact subset of ${\mathbb L}^{2}[(0,1)^{d-1}]$ such that for all $f\in {\cal F}$, $\int_{(0,1)^{d-1}}f(u)du=1$, and real numbers $R_{\text{min}}$ and $R_{\text{max}}$ such that $0<R_{\text{min}}<R_{\text{max}}<+\infty$. Since we shall study minimax rates in Section \ref{sec:rates}, we shall fix later ${\cal F}$ to include all Sobolev classes of interest in that paper. Then we define $(\widehat{f},\widehat{R})$ as any measurable random variable such that
\begin{equation}
\label{eq:def:estimateur1}
M_{n}\left(\widehat{f},\widehat{R}\right) \leq
\inf_{(f,R)\in {\cal F}\times [R_{\text{min}};R_{\text{max}}]} M_{n}\left(f,R\right)+\frac{1}{n}.
\end{equation}
Notice that we do not constrain functions in ${\cal F}$ to be non negative, that is we do not constrain $\widehat{f}$ to be a probability density. Using Proposition \ref{prop:phihat} we get the following corollary which will be the basic stone to obtain estimation rates of our estimators. For any $\nu >0$, define $c_{\nu}^{\star}$ as in \eqref{eq:cnu} with $\mathbb{Q}=\mathbb{Q}^{\star}$. 

\begin{corollary}
\label{cor:basic}
Assume $f^{\star}\in{\cal F}$, $R^{\star}\in (R_{\text{min}};R_{\text{max}})$ and  $\varepsilon_1$ has finite variance.
For all $\nu\in(0,\nu_\text{est}]$ such that $c_{\nu}^{\star}>0$,
for all $\delta >0$, there exist positive constants $c_{1}, c_{2}, c_{3}$ which depend on $\delta$, $\nu$, $c_{\nu}^{\star}$, $d$ and $E(\|Y\|^2)$ such that for all $x \geq 1$ and $n\geq (1\vee xc_{1})/c_{2}$,
with probability at least $1 - e^{-x}$,
$$
\int_{[-\nu,\nu]^{d}} |\Psi_{\widehat{f},\widehat{R}} (t)- \Psi_{f^{\star},R^{\star}}(t)|^{2} dt \leq c_{3}\left(\frac{x}{n^{1-\delta}} \vee \frac{x^2}{n^{2-2\delta}} \right).
$$
\end{corollary}
We insist on the fact that the quantity $c_{\nu}^{\star}$ is unknown, and that its knowledge is not needed to construct the estimators and to get asymptotic rates, since there always exists a small enough $\nu$ such that $c_{\nu}^{\star}>0$. 

When $\mathbb{Q}^{\star}$ has a finite first moment and is centered, we can estimate the center of the sphere. We define
\begin{equation}
\label{eq:def:estimateur2}
\widehat{C}=\frac{1}{n} \sum_{i=1}^{n} Y_{i}-\widehat{R}\int_{(0,1)^{d-1}}S(u)
\widehat{f}(u)du.
\end{equation}
The estimators of the radius and of the exploration density can be proved to be consistent by applying $M$-estimator general results. Then consistency of the estimator of the radius follows. We give a detailed proved of the following proposition in Section \ref{subsec:consi}. Here, it is not needed that the noiuse has finite variance. 

\begin{proposition}
\label{prop:consi}
Assume $f^{\star}\in{\cal F}$ and $R^{\star}\in (R_{\text{min}};R_{\text{max}})$. Then 
$\widehat{R}=R^{\star}+o_{\mathbb{P}_{C^{\star},R^{\star},f^{\star},\mathbb{Q}^{\star}}}(1)$ and
$\int_{(0,1)^{d-1}}(\widehat{f}(u)-f^{\star}(u))^{2}du=o_{\mathbb{P}_{C^{\star},R^{\star},f^{\star},\mathbb{Q}^{\star}}}(1)$. If moreover $\mathbb{Q}^{\star}$ has finite first moment and is a centered distribution, then also $\widehat{C}=C^{\star}+o_{\mathbb{P}_{C^{\star},R^{\star},f^{\star},\mathbb{Q}^{\star}}}(1)$.
\end{proposition}

\section{Convergence rates of the estimators}
\label{sec:rates}

In this section, we prove that the estimator of the radius has almost parametric rate of convergence, whatever the dimension $d$ of the sphere. We then get rates of convergence for the estimator of the exploration density and of the center in the case $d=2$ that is for circular signals. Our estimator of the exploration density achieves the minimax rate on Sobolev regularity classes and the estimator of the center can be proved to have almost parametric rate.

\subsection{The estimator of the radius}
\label{subsec:rad}
Our first main result is the fact that, without any knowledge of the noise distribution and of the exploration density, the radius of the sphere can be recovered at almost parametric rate.
\begin{theorem}
\label{theo:radius}
Assume $f^{\star}\in{\cal F}$ and $R^{\star}\in (R_{\text{min}};R_{\text{max}})$. Assume also that $\varepsilon_1$ has finite variance.
For all $\nu\in(0,\nu_\text{est}]$ such that $c_{\nu}^{\star}>0$,
for all $\delta >0$, there exist positive constants $c_{1}, c_{2}, c_{3}$ which depend on $\delta$, $\nu$, $c_{\nu}^{\star}$, $d$ and $E(\|Y\|^2)$ such that for all $x \geq 1$ and $n\geq (1\vee xc_{1})/c_{2}$,
with probability at least $1 - e^{-x}$,
$$ |\widehat{R} - R^{\star}|^2 
\leq \frac{c_{3}}{R_{\text{min}}^2(1-\frac{(\nu R_{\text{max}})^2}{2d+8})^2}\left(\frac{x}{n^{1-\delta}} \vee \frac{x^2}{n^{2-2\delta}} \right).$$
\end{theorem}


\begin{proof}

We denote by $\Delta$ the Laplacian operator in the Cartesian coordinate system,
$$\Delta = \sum_{i=1}^d \frac{\partial^2}{\partial x_i^2},$$ and for all $k \geq 1$, \  $\Delta^k = \Delta^{k-1} \circ \Delta$ with $\Delta^0$ the identity operator.
\\
Notice that $\Psi_{f,R}$ is an eigenfunction of the Laplacian, with eigenvalue $R^2$,
$$ \Delta \Psi_{f,R}(x) + R^2 \  \Psi_{f,R}(x) = 0,$$
so that for all $k \geq 1$,
$$\Delta^k \Psi_{f,R}(x) + (-1)^{k-1} R^{2k} \  \Psi_{f,R}(x) = 0.$$
Define $\mathbb{D}_{d}(0,\nu)=\{x\in\R^d,\;\|x\|_{2}\leq \nu\}$ the $d$-dimensional disk centered at the origin and with radius $\nu$, $\Gamma$ the gamma function and $\lambda_{d}$ the $d$-dimensional Lebesgue measure.
Then, according to \cite{Laplacian}, for all $\nu > 0$ and for all multivariate analytic function  $\psi$ on $\mathbb{C}^d$,

\begin{eqnarray*} 
\frac{1}{\lambda_d(
\mathbb{D}_d(0,\nu))}  \int_{\mathbb{D}_d(0,\nu)} \psi(x) dx  &= &\sum_{k=0}^{\infty} \frac{\Delta^k \psi(0)}{2^{k}k!\prod_{j=1}^k (d+2j)} \nu^{2k}\\
& = &\Gamma\left(\frac{d}{2}+1\right) \sum_{k=0}^{\infty} \frac{\Delta^k \psi(0)}{2^{2k}k! \Gamma(\frac{d}{2}+k+1)} \nu^{2k}.
\end{eqnarray*}
Applying this equality to  ${\Psi}_{\widehat{f},\widehat{R}}$ and to $\Psi_{f^{\star},R^{\star}}$
we get that
\begin{equation}
\label{eq:R:1}
 \frac{1}{(\sqrt{\pi} \nu)^d} \int_{\mathbb{D}_d(0,\nu)} \left({\Psi}_{\widehat{f},\widehat{R}}(x) - \Psi_{f^{\star},R^{\star}} (x)\right) dx = \sum_{k=0}^{\infty} (-1)^k \frac{\widehat{R}^{2k}- (R^{\star})^{2k}}{2^{2k}k!\Gamma(\frac{d}{2} + k + 1)} \nu^{2k},
 \end{equation}
since $\lambda^d(\mathbb{D}_{d}(0,\nu)) = \frac{\pi^{d/2}}{\Gamma(\frac{d}{2}+1)} \nu^{d}$. \\
Let
 $J_{d/2}$ be the Bessel function of order $(d/2)$. 
 We collect in Section \ref{sec:Bessel}  results on Bessel functions that will be useful in our analysis. Using identity (I) in in Section \ref{sec:Bessel} we get
 that for all $x\in\R$, 
\begin{equation}
\label{eq:R:2}
\sum_{k=0}^{\infty} (-1)^k \frac{x^{2k}}{2^{2k}k!\Gamma(\frac{d}{2} + k + 1)} = 2^{d/2} \frac{J_{d/2}(x)}{x^{d/2}}, 
 \end{equation}
so that
 using \eqref{eq:R:1}, \eqref{eq:R:2}, Cauchy-Schwarz inequality and the fact that 
$\mathbb{D}_d(0,\nu)\subset B_{\nu}^d$ we obtain
\begin{equation}
\label{eq:R:3}
2^{d} \left \vert \frac{J_{d/2}(\nu \widehat{R})}{(\nu \widehat{R})^{d/2}} - \frac{J_{d/2}(\nu R^{\star})}{(\nu R^{\star})^{d/2}} \right \vert^{2} \leq  \frac{1}{(\sqrt{\pi} \nu)^{d} \Gamma(\frac{d}{2} +1)}  \int_{B_{\nu}^d}\left( {\Psi}_{\widehat{f},\widehat{R}}(x) - \Psi_{f^{\star},R^{\star}}(x) \right)^{2}dx.
 \end{equation}
Let $H$ be the function defined by
$$
\forall x \neq 0,\;H(x) = \frac{J_{d/2}( x)}{x^{d/2}},\; H(0) = \frac{1}{2^{d/2} \Gamma(\frac{d}{2} + 1)}.$$
Then using \eqref{eq:R:2}, $H$ has infinitely many derivatives so that  there exists $\widetilde{R}\in (R^{\star},\widehat{R})$   
such that
$$
H(\nu \widehat{R}) - H(\nu R^{\star}) = \nu (\widehat{R} - R^{\star}) H'(\nu \widetilde{R}).
$$
Computation of the derivative and (IV) in Section \ref{sec:Bessel} gives
$$H'(x) = \frac{ x (J_{d/2 -1}(x) -  J_{d/2 + 1}( x)) - d J_{d/2}( x)}{2 x^{d/2+1}},$$
and using (V) in Section \ref{sec:Bessel} we get 
$$H'(x) = -  \frac{J_{d/2+1}( x)}{  x^{d/2}}.
$$
Using lemma \ref{bessel1} in Section \ref{sec:Bessel}, we get that
$$ J_{d/2 + 1}(\nu \widetilde{R}) \geq \frac{(\nu \widetilde{R})^{d/2 + 1}}{2^{d/2 + 1} \Gamma(\frac{d}{2} + 2)}\left(1-\frac{(\nu \widetilde{R})^2}{2d + 8}\right). $$
Since $\widetilde{R} \in (R_{min},R_{max})$, we deduce that for any $\nu\in (0,1/R_{max})$, 

\begin{align*}
|H'( \nu \widetilde{R})| &= \left| \frac{J_{d/2}( \nu \widetilde{R})}{(\nu \widetilde{R})^{d/2}} \right|  \geq  \frac{\nu \widetilde{R}}{2^{d/2 + 1} \Gamma(\frac{d}{2} + 2)} \left| 1 - \frac{(\nu \widetilde{R})^2}{2d+8} \right|\\
    & \geq \frac{ \nu R_{min}}{2^{d/2 + 1} \Gamma(\frac{d}{2} + 2)}\left(1-\frac{(\nu R_{max})^2}{2d + 8}\right) >  0,
 \end{align*}

so that
$$|\widehat{R} - R^{\star}| \leq \frac{2^{d/2 + 1} \Gamma(\frac{d}{2}+2)}{\nu^2 R_{min}(1-\frac{(\nu R_{max})^2}{2d+8})} |H(\nu \widehat{R}) - H(\nu R^{\star})|.$$
Using \eqref{eq:R:3}
we get
$$ |\widehat{R} - R^{\star}|^2 \leq  \frac{ 4 \Gamma(\frac{d}{2} + 2)^2}{\nu^4 (\sqrt{\pi} \nu) ^{d} R_{min}^2(1-\frac{(\nu R_{max})^2}{2d+8})^2} \int_{B_{\nu}^d}\left( {\Psi}_{\widehat{f},\widehat{R}}(x) - \Psi_{f^{\star},R^{\star}}(x) \right)^{2}dx.$$
The end of the proof follows from
Corollary \ref{cor:basic}.
\end{proof}

\subsection{The estimator of the density and of the center}
\label{subsec:dens}
In this section, we consider the case of circular signals, that is $d=2$.
In this case, we can rewrite model \eqref{eq:sphere} using one dimensional angles $U_{i}\in [0,1]$, as
\begin{equation}
\label{eq:circle}
X_{i}=C^{\star} + R^{\star}\begin{pmatrix} \cos(2 \pi U_{i}) \\  \sin(2\pi U_{i})  \end{pmatrix}.
\end{equation}
We shall focus on the following regularity classes. For any $L>0$, $\beta > \frac{1}{2}$ and $\gamma > 0$, set
$$ W_{\beta}(L) = \lbrace f \in \mathbb{L}^2([0,1]) : \sum_{k = - \infty}^{\infty} |f_k|^2 |k|^{2 \beta} \leq L^2 \rbrace,$$
$$A_{\gamma}(L) = \lbrace f \in \mathbb{L}^2([0,1]) : \sum_{k = - \infty}^{\infty} |f_k|^2 e^{2 \gamma k} \leq L^2 \rbrace,$$
where for any function $f\in \mathbb{L}^2([0,1])$, $(f_k)_{k\in \Z}$ is the sequence of Fourier coefficients of $f$:
$$ f_k = \int_0^{1} f(\theta) e^{2 i \pi k \theta} d\theta, \ \ k \in \mathbb{Z}.$$
We fix $\cal F$ as a compact subset of ${\mathbb L}^{2}[(0,1)^{d-1}]$ such that for all $f\in {\cal F}$, $\int_{(0,1)^{d-1}}f(u)du=1$, and containing as subsets all 
$W_{\beta}(L)$ and $A_{\gamma}(L)$ for all $\beta > \frac{1}{2}$, $\gamma > 0$, and $L\leq L_{max}$ chosen.
We shall now define an estimator of $f^{\star}$ using truncated Fourier expansions of $\widehat{f}$ defined in Section \ref{subsec:method}. 
For $N > 0$ an integer to be chosen, we define $T_N\widehat{f}$ the trigonometric polynomial estimator of $f^{\star}$:
$$ 
\forall x \in (0,1), \ \ T_N\widehat{f}(x) = \sum_{|k| \leq N} \widehat{f}_k e^{-2 i \pi k x}.
$$

For any $\nu>$, $c_{\nu} >0$, $E>0$, define
${\cal Q} (\nu,c({\nu}),E)$ the set of distribution $\mathbb{Q}=\mathbb{Q}_{1}\otimes \mathbb{Q}_{2}$ on $\R^2$ such that $c_{\nu}\geq c(\nu)$ and $\int_{\R^2} \|x\|^2 d\mathbb{Q}(x)\leq E$. Define now the maximum risk of the estimator for any class of densities $\cal C$ and any class of noise distribution $\cal Q$ as follows.
$$
R\left[T_{N}\widehat{f}; {\cal C}; R_{\text{min}};R_{\text{max}}; {\cal Q}\right]
=\sup_{f\in {\cal C}, R\in [R_{\text{min}};R_{\text{max}}], {\mathbb Q}\in {\cal Q}, C\in \R^2}
\mathbb{E}_{C,R,f,{\mathbb Q}}\left(\int_{0}^{1}(T_{N}\widehat{f} (x) - f(x))^2dx\right).
$$
The following theorem shows that a good choice of $N$ leads to minimax adaptive estimation rate over the regularity classes $W_{\beta}(L)$ and controlled maximum risk over the regularity classes $A_{\gamma}(L)$.
\begin{theorem}
\label{theo:density}
For $\alpha \in (0,1/2)$, set
 $$N= \left \lfloor \alpha\frac{\log n}{\log \log n}
 \right \rfloor.  $$ Then for all $L>0$, $\beta > 1/2$, (resp. $\gamma >0$), for all $\nu\in(0,\nu_\text{est}]$,
 $c(\nu)>0$, $E>0$,
\begin{equation}
\label{eq:beta}
R\left[T_{N}\widehat{f}; W_{\beta}(L); R_{\text{min}};R_{\text{max}}; {\cal Q}(\nu,c({\nu}),E)\right]
\leq L^{2}\alpha^{-2\beta}\left ( \frac{lnln(n)}{ln(n)} \right )^{2\beta}(1+o(1))
\end{equation}
as $n$ tends to infinity, and
\begin{equation}
\label{eq:gamma}
R\left[T_{N}\widehat{f}; A_{\gamma}(L); R_{\text{min}};R_{\text{max}}; {\cal Q}(\nu,c({\nu}),E)\right]
\leq exp \left ( -2 \gamma \left (\alpha\frac{ln(n)}{lnln(n)} \right ) \right )(1+o(1)) 
\end{equation}
as $n$ tends to infinity.
\end{theorem}
In \cite{Go2002}, the author studies the estimation of the exploration density for noisy circular data on the unit circle (known radius) and with known noise distribution. Comparison of \eqref{eq:beta} with Theorem 1 in \cite{Go2002} shows that our estimator is rate minimax adaptive to unknown radius, unknown noise distribution and unknown regularity over classes $W_{\beta}(L)$ for the signal, with a constant deteriorated by a factor at most $2^{2\beta}$. Comparing \eqref{eq:gamma} with Theorem 2 in \cite{Go2002} shows that a loss in the upper bound for the rate of convergence of the maximum risk of our estimator in case of unknown radius and unknown noise distribution on classes $A_{\gamma}(L)$ for the signal.

\begin{proof}

For any $f\in {\mathcal F}$,
$$ \int_{0}^{1}\left(T_N\widehat{f}(x) - f(x)\right)^2dx = \sum_ {|k| \leq N} |f_k - \widehat{f}_k|^2 + \sum_{|k|>N} |f_k|^2,$$
so that for any $R\in (R_{\text{min}};R_{\text{max}})$, ${\mathbb Q}\in {\cal Q}(\nu,c({\nu}),E)$, $C\in \R^2$,
$$
\mathbb{E}_{C,R,f,{\mathbb Q}}\left(\int_{0}^{1}(T_{N}\widehat{f} (x) - f(x))^2dx\right) = \mathbb{E}_{C,R,f,{\mathbb Q}}\left[\sum_ {|k| \leq N} |f_k - \widehat{f}_k|^2\right] + \sum_{|k|>N} |f_k|^2.
$$
The first term on the right hand side will be shown to be negligible with respect to the second term thanks to the following proposition, for which a detailed proof can be found in Section \ref{subsec:propdensity}
\begin{proposition}
\label{prop:density}
Assume $f^{\star} \in {\cal F}$ and $R^{\star}\in (R_{\text{min}},R_{\text{max}})$. For all $\nu\in(0,\nu_\text{est}]$ such that $c_{\nu}^{\star}>0$,
for all $\delta >0$, there exists a constant $ c > 0$ depending on $\delta$, $\nu$, $c_{\nu}^{\star}$, $d$, $R^{\star}$, $R_{\text{min}}$, $R_{\text{max}}$, and $E(\|Y\|^2)$ such that for all $x \geq 1$, 
 $n\geq (1\vee xc_{1})/c_{2}$,
with probability at least $1 - e^{-x}$,
\begin{equation} \sum_{|k| \leq N} |f_k^{\star} - \widehat{f}_k|^2 \leq c \left ( \frac{2}{\nu R_{\text{min}}} \right )^{2N}(N+1) (N!)^2  \left(\frac{x}{n^{1-\delta}} \vee \frac{x^2}{n^{2-2\delta}} \right).
\label{eq:density}
\end{equation}
\end{proposition}
Choose $\delta$ small enough so that $2\alpha < 1-\delta$.
Then for large enough $n$, for a constant $c>0$,
$$
\mathbb{E}_{C,R,f,{\mathbb Q}}\left[\sum_ {|k| \leq N} |f_k - \hat{f}_k^n|^2\right] \leq c \ (\nu R_{\text{min}})^{-2N} (N+1) 2^{2N} (N!)^2 n^{-1 + \delta},
$$
and using the fact that, $\forall N \geq 1$, $N! \leq e N^{N + \frac{1}{2}} e^{-N}$, we finally have,
\begin{equation}
\label{eq:var}
\mathbb{E}_{C,R,f,{\mathbb Q}}\left[\sum_ {|k| \leq N} |f_k - \hat{f}_k^n|^2\right] \leq e^2 c \ (\nu R_{\text{min}})^{-2N} (N+1) 2^{2N} N^{2N+1}e^{-2N}  n^{-1 + \delta}.
\end{equation}

The term at the right hand side of \eqref{eq:var} is at most of order

$$ \exp \Bigg \{ (2 \alpha + \delta - 1) \ ln(n) \Bigg [ 1 + o(1) \Bigg ] \Bigg \}. $$

Now,
\begin{equation}
\label{eq:bias1}
\sup_{f\in W_{\beta}(L)}\sum_{|k|>N} |f_k|^2 \leq L^2 N^{-2 \beta},
\end{equation}
and
\begin{equation}
\label{eq:bias2}
\sup_{f\in A_{\gamma}(L)}\sum_{|k|>N} |f_k|^2 \leq
L^2 e^{-2 \gamma N}.
\end{equation}
Equation \eqref{eq:beta} follows from \eqref{eq:var} and \eqref{eq:bias1}, and equation \eqref{eq:gamma} follows from \eqref{eq:var} and \eqref{eq:bias2}.
\end{proof}

\begin{theorem}
\label{theo:center}
Assume $f^{\star} \in {\cal F}$ and $R^{\star}\in (R_{\text{min}},R_{\text{max}})$. Assume also that  $\varepsilon_1$ has finite variance.
Then for any $\delta >0$, 
$$\left( \widehat{C}-C^{\star} \right)
=O_{\mathbb{P}_{C^{\star},R^{\star},f^{\star},\mathbb{Q}^{\star}}}\left(n^{-1/2 + \delta}\right).
$$

\end{theorem}
Notice that we can not get exponential deviations for the empirical mean of the observations when nothing more is assumed about the noise apart having finite variance. Since the estimator of the center involves the empirical mean of the observations, we only prove tightness of $n^{1/2 - \delta}\left( \widehat{C}-C^{\star} \right)$.
\begin{proof}
Notice that 
$$
f^{\star}_{1}=\int_{0}^{1}\cos (2\pi u) f^{\star}(u) du + i \int_{0}^{1}\sin (2\pi u) f^{\star}(u) du
$$
so that
$$C^{\star} = \mathbb{E}[Y] -R^{\star}\begin{pmatrix} \text{Re}(f^{\star}_1) \\ \text{Im}(f^{\star}_1) \end{pmatrix},
$$
and in the same way
$$\widehat{C} = \frac{1}{n} \sum_{l=1}^n Y_l - \widehat{R} \begin{pmatrix} \text{Re}(\widehat{f}_1) \\ \text{Im}(\widehat{f}_1) \end{pmatrix}.
$$
Thus, using the triangle inequality, and the fact that
 $\widehat{R} \leq R_{\text{max}}$ and $|f^{\star}_{1}|\leq 1$,
 we get
$$\|\widehat{C} - C^{\star} \|_2 \leq \left\|\frac{1}{n} \sum_{l=1}^n Y_l - \mathbb{E}[Y] \right\|_2 + 
| \widehat{R} - R^{\star} | + R_{\text{max}} \left\| \begin{pmatrix} \text{Re}(f^{\star}_1) - \text{Re}(\widehat{f}_1)  \\ \text{Im}(f^{\star}_1) - \text{Im}(\widehat{f}_1) \end{pmatrix} \right\|_2.$$
The theorem follows from Theorem \ref{theo:radius}, Proposition \ref{prop:density} and the central limit theorem.

\end{proof}

\section{When the exploration density is known}
\label{sec:semipara}
In this section, we assume that $f^{\star}$ is known. By exchanging the role of the signal and of the noise, we can look at model \eqref{eq:model} as a semi-parametric deconvolution problem in which the noise has known distribution (up to centering and radius) on a sphere. But we are able to estimate the radius and the center without solving the semi-parametric deconvolution problem, that is without estimating $\mathbb{Q}$. We estimate the radius using the contrast function
$M_{n}(f^{\star},R)$. Since this function is continuous, we can define
$$
\widetilde{R}= {\text{ Argmin}} \;\{M_{n}(f^{\star},R),\;R\in [R_{\text{min}};R_{\text{max}}]\}.
$$
If moreover $\mathbb{Q}^{\star}$ has finite first moment and is a centered distribution, then the estimator of the center is defined as
$$
\widetilde{C}=\frac{1}{n} \sum_{i=1}^{n} Y_{i}-\widetilde{R}\int_{(0,1)^{d-1}}S(u)f^{\star}(u)du.
$$
Theorem \ref{theo:semipara} 
states that $\sqrt{n}(\widetilde{R}-R^{\star},\widetilde{C}-C^{\star})$   converges in distribution as $n$ tends to infinity to some centered Gaussian distribution. It will be a consequence of the lemma stated below.
In the following, for $R \in [R_{\text{min}}, R_{\text{max}}]$, we omit $f^{\star}$ as an argument of $M$ and $M_n$. We write $M'$, $M'_n$ their derivatives with respect to $R$ and $M''$, $M''_n$ their second derivatives with respect to $R$.
%
\begin{lemma}
\label{lem:convgauss}
The following results hold true under the assumptions of Theorem \ref{theo:semipara}. 
\begin{enumerate}
\item[(1)] $\widetilde{R}$ is a consistent estimator of $R^{\star}$.
\item[(2)] There exists a matrix $V$ such that $\sqrt{n}\left(\frac{1}{n}\sum_{i=1}^{n}Y_{i}-E(Y_{1}),  M'_n(R^{\star})\right)$  converges in distribution to a centered Gaussian distribution with variance V.
\item[(3)] $ M''(R^{\star}) \neq 0$ and for any random variable $R_n \in [R_{min},R_{max}]$ converging in probability to $R^{\star}$, one has 
$$M''_n(R_n) = M''(R^{\star}) + o_{\mathbb{P}_{C^{\star},R^{\star},f^{\star},\mathbb{Q}^{\star}}}(1).$$
\end{enumerate}
\end{lemma}
The proof of Lemma \ref{lem:convgauss} is given in Section \ref{subsec:convgauss}.\\
Define the $(d+1)\times (d+1)$ matrix
$$
\Sigma=\begin{pmatrix}0 &-\frac{1}{M''(R^{\star})}\\
1 &\frac{E(S(U))}{M''(R^{\star})}\end{pmatrix}  V \begin{pmatrix}0 & 1\\
-\frac{1}{M''(R^{\star})} &\frac{E(S(U))^T }{M''(R^{\star})} \end{pmatrix}
$$
\begin{theorem}
\label{theo:semipara}
Assume that $\mathbb{Q}^{\star}$ has finite second moment and is a centered distribution.
Then $\sqrt{n}(\widetilde{R}-R^{\star},\widetilde{C}-C^{\star})$ converges in distribution to a centered Gaussian distribution with variance $\Sigma$.
\end{theorem}
The proof of Theorem \ref{theo:semipara} is detailed in Section \ref{subsec:semipara}.

\section{Simulations}
\label{sec:simu}

The aim of this section is to illustrate our method with examples for which the noise is not bounded. We choose $d=2$ and we consider the model \eqref{eq:circle} with $R^{\star} = 3$, $C^{\star} = 0$ and $U, \varepsilon$ generated as follows.

\begin{enumerate}
    \item[(1)]
$U \sim \text{Unif}(0,1)$  and $\varepsilon \sim \mathcal{N}(0,(0.12)^2 I)$, figure \ref{figure1}
        \item[(2)] 
$U \sim \text{Unif}(0,1)$  and for $i \in \lbrace 1,2 \rbrace$  $\varepsilon^{(i)} \sim  \frac{1}{2} \delta_{(-1)} + \frac{1}{2} \mathcal{E}\text{xp}\left\{\frac{1}{0.12}\right\}$, figure \ref{figure2}
    \item[(3)]
$U \sim \text{Unif}(0,1)$  and $\varepsilon \sim \mathcal{N}(0,I)$, figure \ref{figure3}
    \item[(4)]
$U \sim f_U : x \in (0,1) \mapsto \frac{\exp\left\{\cos(2 \pi x)\right\}}{\int_{0}^1\exp\left\{\cos(2 \pi u)\right\} du} $  and $\varepsilon \sim \mathcal{N}(\begin{pmatrix} -1.6 \\ 2.5 \end{pmatrix},\begin{pmatrix} (0.2)^2 & 0 \\ 0 & (0.57)^2 \end{pmatrix})$, figure \ref{figure4}

\end{enumerate}

For each case, we generate $n$ observed points for $n \in V$ with 
$$ V = \lbrace 10^2,2 \cdot 10^2,3 \cdot 10^2 ,4 \cdot 10^2 ,5 \cdot 10^2,  10^3 ,2 \cdot 10^3,3 \cdot 10^3,5 \cdot 10^3, 10^4,5 \cdot 10^4,\\
7.5 \cdot 10^4 , 10^5,3 \cdot 10^5 ,5 \cdot 10^5,8 \cdot 10^5, 10^6 \rbrace.  $$
We estimate the radius of the circle in the case where the exploration density is known, and unknown.\\
In practice, when we want to estimate the radius, the choice of $N$ and $\nu$ does not significantly change the results thus the simulations are done with $N = \left \lfloor{\frac{\text{log}(n)}{\text{log(log}(n))}} \right \rfloor $ and $\nu = 0.5$.\\
For each figure, there are $4$ plots,

\begin{enumerate}
    \item[\textbf{Top left}] Scatter plot of the $10^6$ observed points.
    \item[\textbf{Top right}] Scatter plot of the $10^6$ observed points + the support of the signal $X$.
    \item[\textbf{Bottom left}] Plot of $(\text{log}|\widehat{R} - R^{\star}|, \text{log}(n))_{n \in V}$ + the linear regression, when the density $f^{\star}$ is known and unknown.
    \item[\textbf{Bottom right}] To avoid that the step size in $V$ is not constant and to better visualize the graph, we choose $W \subset V$ and we plot $(\widehat{R},n)_{n \in W}$, when the density $f^{\star}$ is known and unknown.
\end{enumerate}

\begin{figure}[H]
\centering
\includegraphics[width=0.65\textwidth]{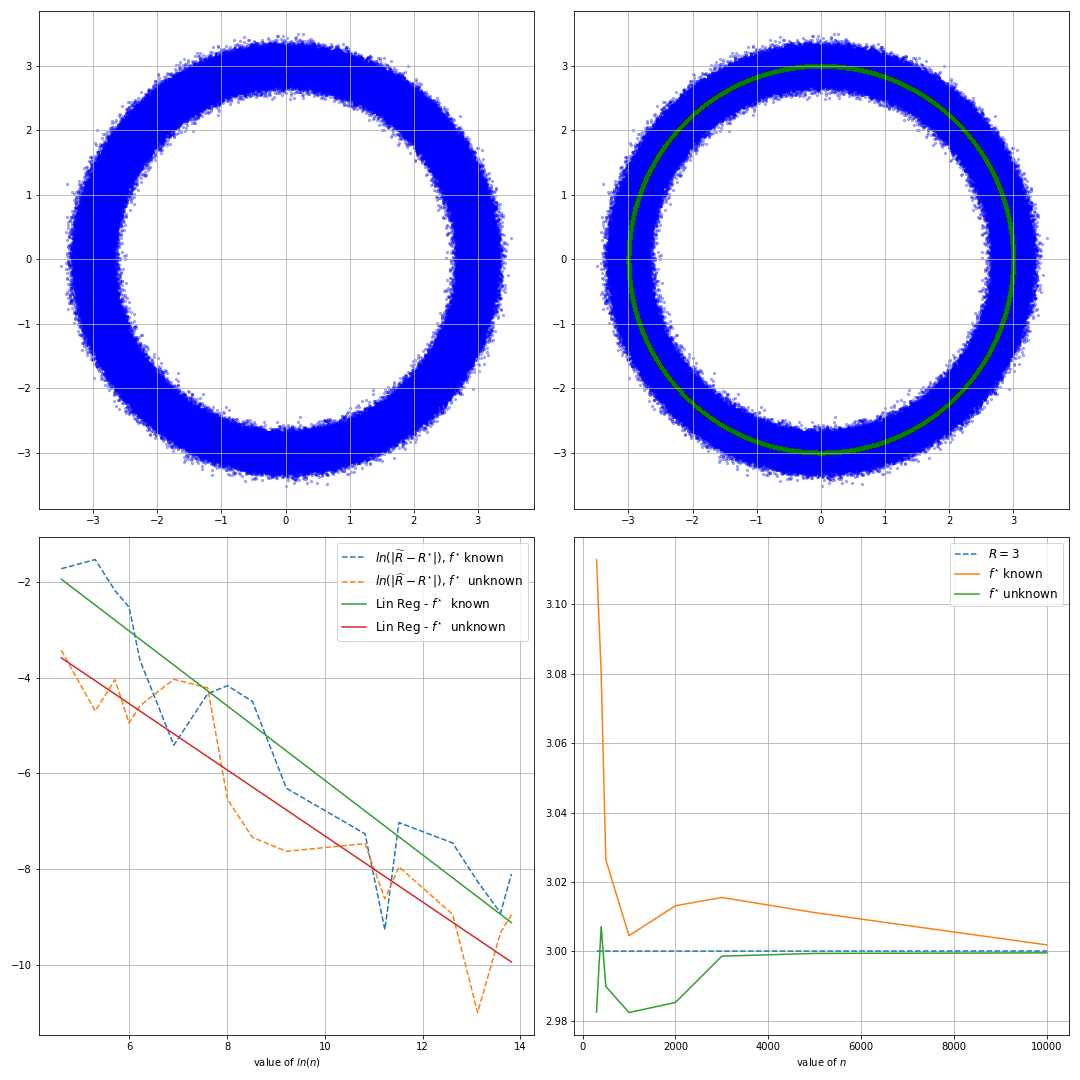}
\caption{\label{figure1} $U \sim \text{Unif}(0,1)$ , $\varepsilon \sim \mathcal{N}(0,(0.12)^2 I)$ and $W = \lbrace 3 \cdot 10^2 ,4 \cdot 10^2 ,5 \cdot 10^2,  10^3 ,2 \cdot 10^3,3 \cdot 10^3,5 \cdot 10^3, 10^4,5 \cdot 10^4, 7.5 \cdot 10^4 , 10^5 \rbrace $ }
\end{figure}

\begin{figure}[H]
\centering
\includegraphics[width=0.65\textwidth]{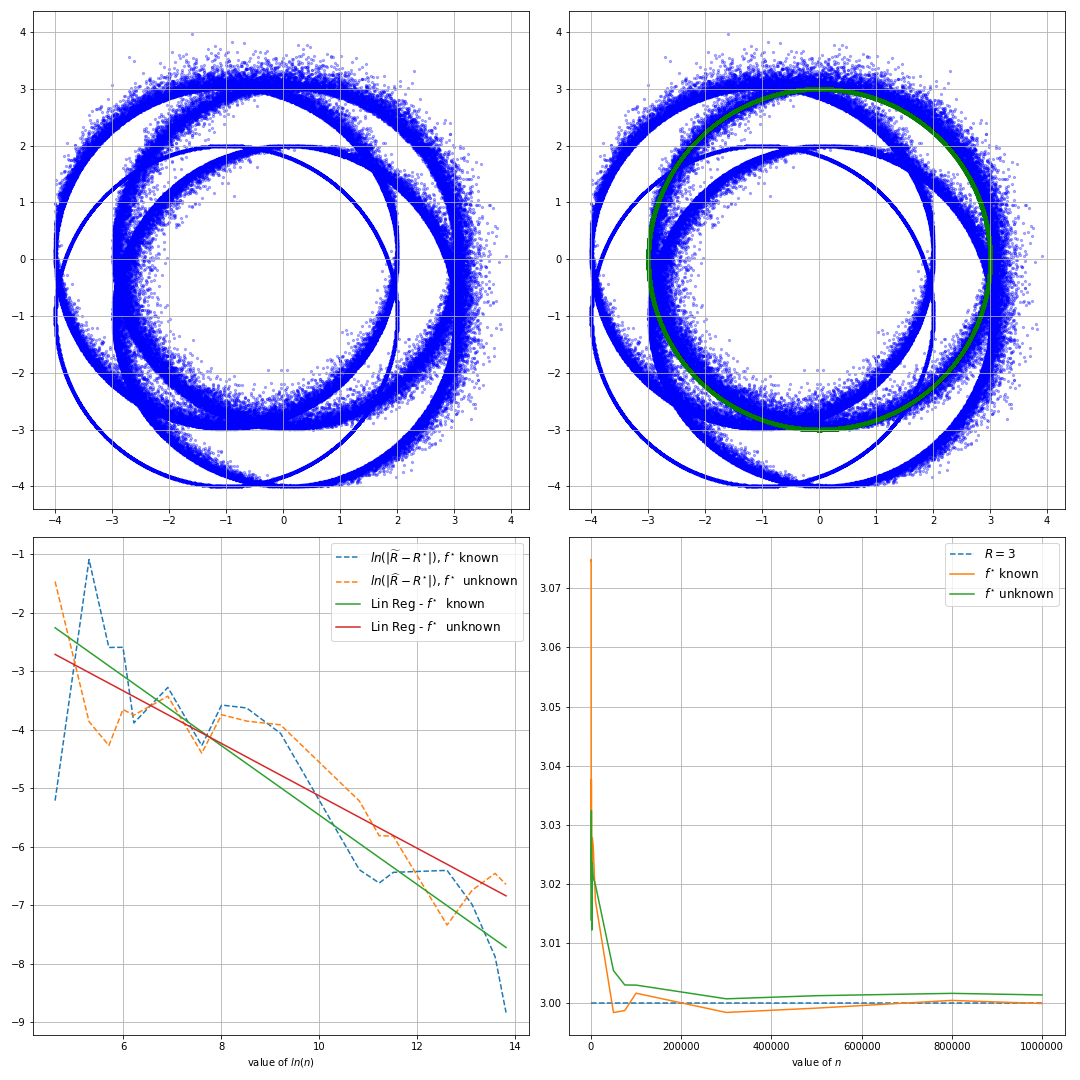}
\caption{\label{figure2} $U \sim \text{Unif}(0,1)$ , for $i \in \lbrace 1,2 \rbrace$  $\varepsilon^{(i)} \sim  \frac{1}{2} \delta_{(-1)} + \frac{1}{2} \mathcal{E}\text{xp}\left\{\frac{1}{0.12}\right\}$  and $W = \lbrace 3 \cdot 10^2 ,4 \cdot 10^2 ,5 \cdot 10^2,  10^3 ,2 \cdot 10^3,3 \cdot 10^3,5 \cdot 10^3, 10^4,5 \cdot 10^4, 7.5 \cdot 10^4 , 10^5,3 \cdot 10^5 ,5 \cdot 10^5,8 \cdot 10^5, 10^6 \rbrace $ }
\end{figure}

\begin{figure}[H]
\centering
\includegraphics[width=0.65\textwidth]{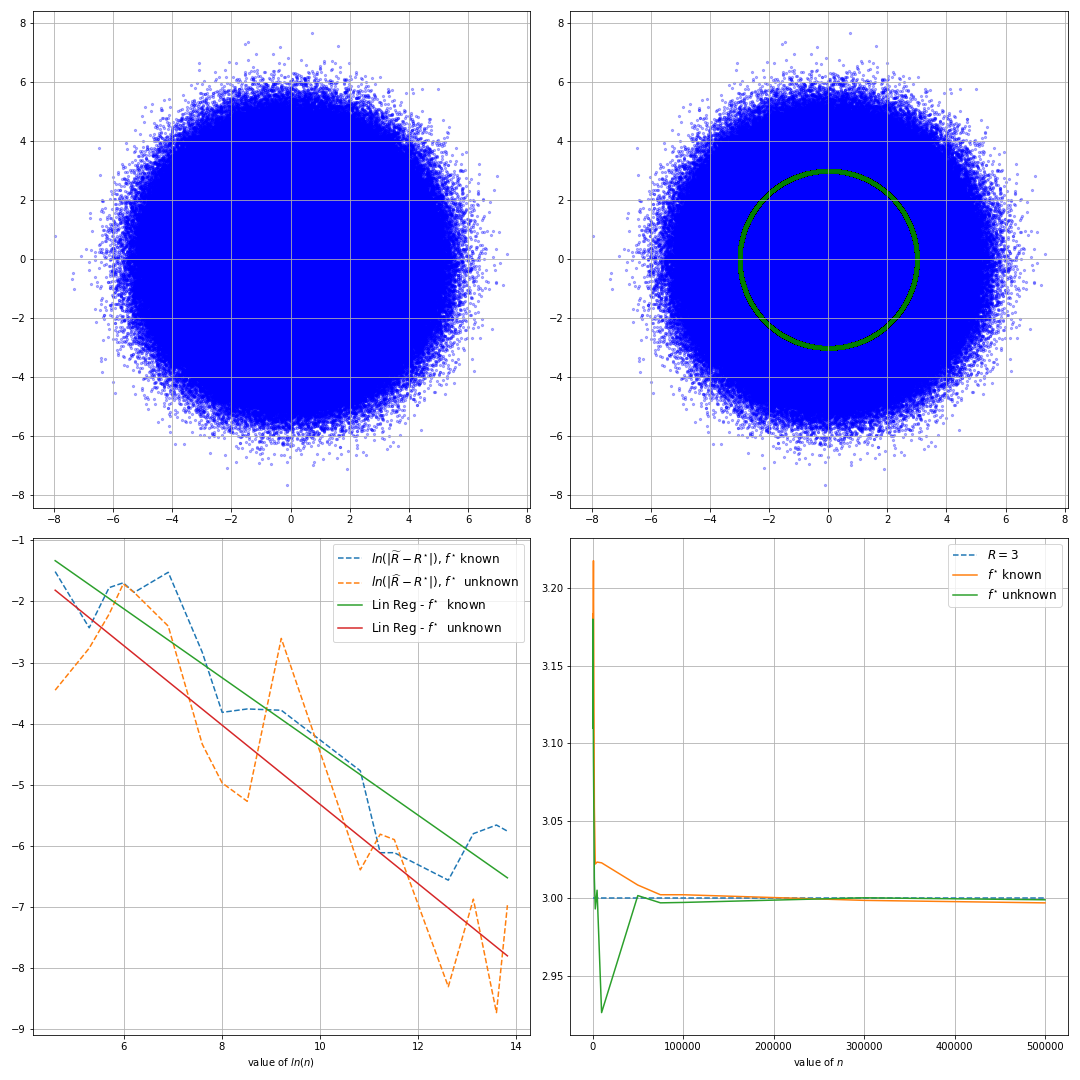}
\caption{\label{figure3} $U \sim \text{Unif}(0,1)$ , $\varepsilon \sim \mathcal{N}(0,I)$ and $W = \lbrace 3 \cdot 10^2 ,4 \cdot 10^2 ,5 \cdot 10^2,  10^3 ,2 \cdot 10^3,3 \cdot 10^3,5 \cdot 10^3, 10^4,5 \cdot 10^4, 7.5 \cdot 10^4 , 10^5,3 \cdot 10^5 ,5 \cdot 10^5 \rbrace$ }
\end{figure}

\begin{figure}[H]
\centering
\includegraphics[width=0.65\textwidth]{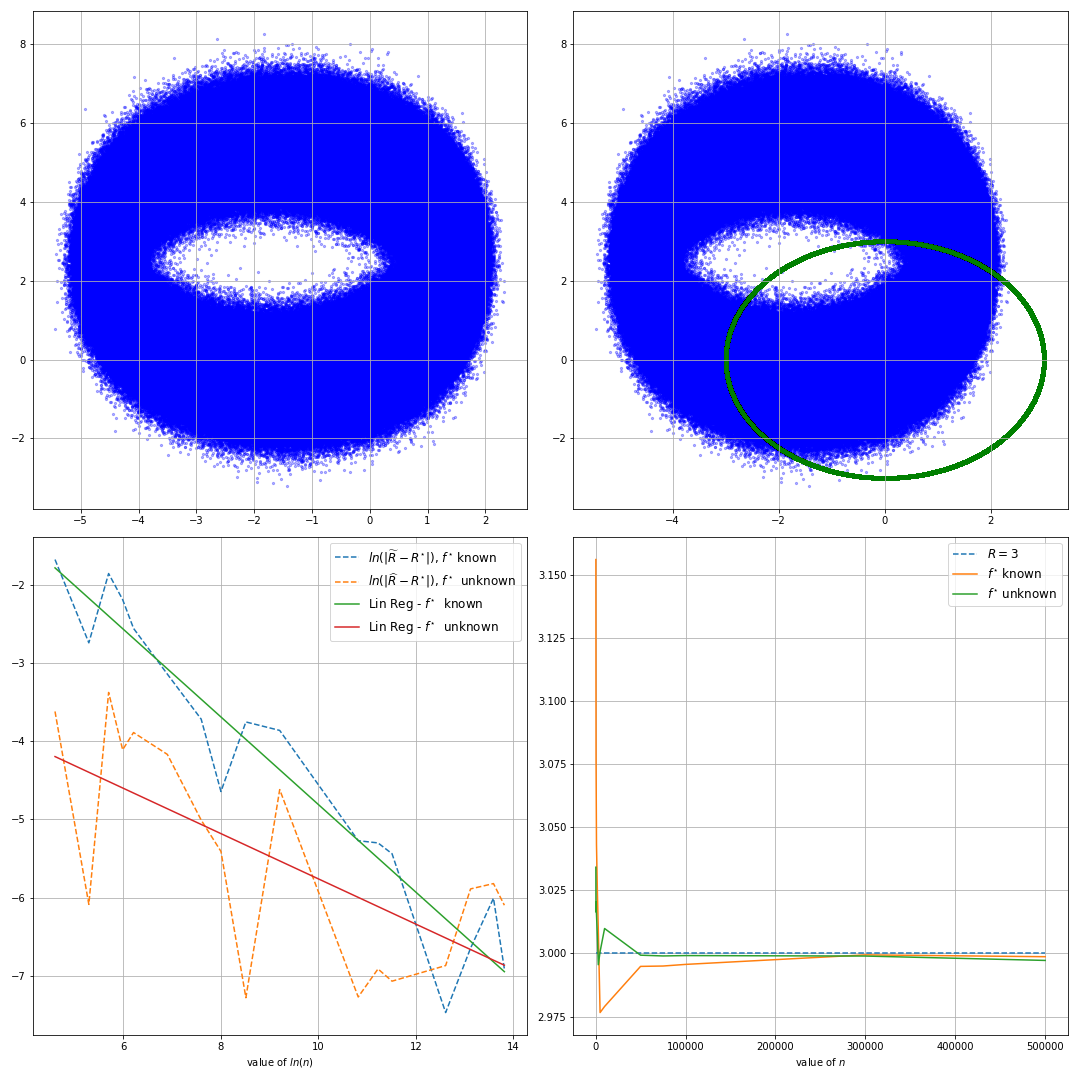}
\caption{\label{figure4} $U \sim f_U : x \in (0,1) \mapsto \frac{\exp\left\{\cos(2 \pi x)\right\}}{\int_{0}^1\exp\left\{\cos(2 \pi u)\right\} du} $ , $\varepsilon \sim \mathcal{N}(\begin{pmatrix} -1.6 \\ 2.5 \end{pmatrix},\begin{pmatrix} (0.2)^2 & 0 \\ 0 & (0.57)^2 \end{pmatrix})$ and $W = \lbrace 3 \cdot 10^2 ,4 \cdot 10^2 ,5 \cdot 10^2,  10^3 ,2 \cdot 10^3,3 \cdot 10^3,5 \cdot 10^3, 10^4,5 \cdot 10^4, 7.5 \cdot 10^4 , 10^5,3 \cdot 10^5 ,5 \cdot 10^5 \rbrace$ }
\end{figure}

The graph of $log|R^{\star} - \widehat{R}|$ of figure \ref{figure1}, \ref{figure2} and \ref{figure3} drive us to reasonably conjecture that the rate of convergence of $|R^{\star} - \widehat{R}|$ is the same when the density $f^{\star}$ is known and unknown.\\
We use Monte-Carlo to estimate $M_n$ and the package \textit{optimize.minimize} in \textit{Python} to minimize $M_n$, that is, there is possibly a numerical bias that can explain the fluctuations on the values of $\widehat{R}$ as we can see in the figure \ref{figure5} and figure \ref{figure6}. (left histogram) 
The histograms are computed with Monte-Carlo replications of $50$ values of $\widehat{R}$ for $n = 10.000$ in the case of figure \ref{figure1} and figure \ref{figure2}.

\begin{figure}[H]

\includegraphics[width=1.2\textwidth]{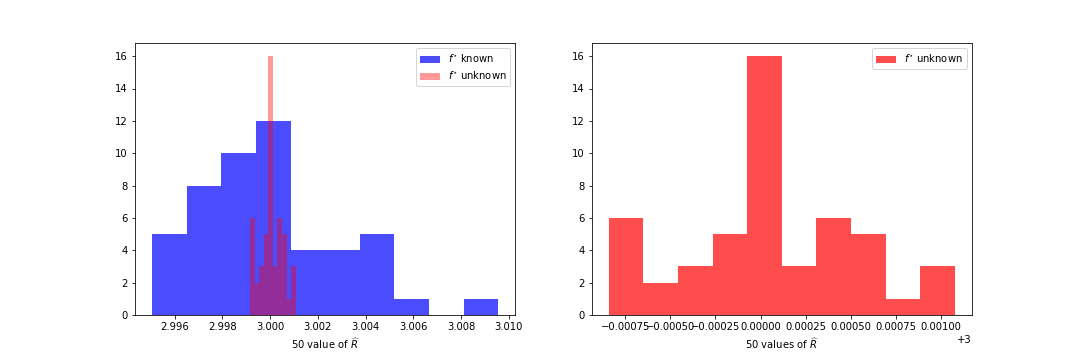}
\caption{\label{figure5} $U \sim \text{Unif}(0,1)$  and $\varepsilon \sim \mathcal{N}(0,(0.12)^2 I)$}
\end{figure}

\begin{figure}[H]
\begin{center}
\includegraphics[width=0.6\textwidth]{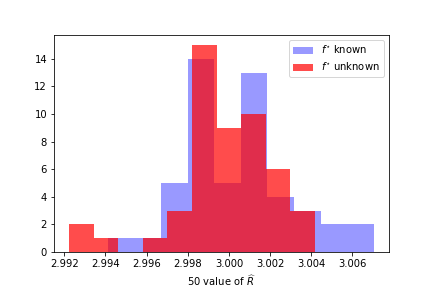}
\caption{\label{figure6} $U \sim \text{Unif}(0,1)$  and for $i \in \lbrace 1,2 \rbrace$  $\varepsilon^{(i)} \sim  \frac{1}{2} \delta_{(-1)} + \frac{1}{2} \mathcal{E}\text{xp}\left\{\frac{1}{0.12}\right\}$}
\end{center}
\end{figure}

Finally, for each $n \in V$ in the case of figure \ref{figure1}, we computed $30$ values of $\widehat{R}$ denoted by $(\widehat{R}^n_k)_{ 1 \leq k \leq 30}$ when the density $f^{\star}$ is known and unknown, we give the following table which gives the emperical mean squarred error.

\begin{center}
\renewcommand\arraystretch{1.4}
\[
\begin{tabular}{|c|c|c|}
  \hline
  \multirow{3}{*} {$n$} & $f^{\star}$ known & $f^{\star}$ unknown \\
  \cline{2-3}
  & $\frac{1}{30} \sum_{k=1}^{30} |R^{\star} - \widetilde{R}^n_k|^2$ & $ \frac{1}{30} \sum_{k=1}^{30} |R^{\star} - \widehat{R}^n_k|^2$ \\
  \hline
  $10^2$ & $7.09 \cdot 10^{-2}$ & $2.86 \cdot 10^{-3}$ \\
  $2 \cdot 10^2$ & $3.87 \cdot 10^{-2}$ & $1.02 \cdot 10^{-3}$ \\
  $3 \cdot 10^2$ & $1.87 \cdot 10^{-2}$ & $5.72 \cdot 10^{-4}$ \\
  $4 \cdot 10^2$ & $1.5 \cdot 10^{-2}$ & $6.15 \cdot 10^{-4}$ \\
  $5 \cdot 10^2$ & $7.28 \cdot 10^{-3}$ & $4.24 \cdot 10^{-4}$ \\
  $10^3$ & $1.27 \cdot 10^{-3}$ & $2.07 \cdot 10^{-4}$ \\
  $2 \cdot 10^3$ & $3.64 \cdot 10^{-4}$ & $9.97 \cdot 10^{-5}$ \\
  $3 \cdot 10^3$ & $1.39 \cdot 10^{-4}$ & $8.73 \cdot 10^{-5}$ \\
  $5 \cdot 10^3$ & $1.31 \cdot 10^{-4}$ & $5.29 \cdot 10^{-5}$ \\
  $10^4$ & $6.99 \cdot 10^{-5}$ & $6.63 \cdot 10^{-6}$\\
  $5 \cdot 10^4$ & $1.01 \cdot 10^{-5}$ & $4.96 \cdot 10^{-7}$ \\
  $7.5 \cdot 10^4$ & $7.75 \cdot 10^{-6}$ & $2.72 \cdot 10^{-7}$ \\
  $10^{5}$ & $7.63 \cdot 10^{-6}$ & $1.96 \cdot 10^{-7}$ \\
  $3 \cdot 10^5$ & $2.47 \cdot 10^{-6}$ & $6.03 \cdot 10^{-8}$ \\
  $5 \cdot 10^5$ & $1.88 \cdot 10^{-6}$ & $3.27 \cdot 10p^{-8}$ \\
  $8 \cdot 10^5$ & $1.37 \cdot 10^{-6}$ & $2.32 \cdot 10^{-8}$ \\
  $10^6$ & $1.02 \cdot 10^{-6}$ & $1.78 \cdot 10^{-8}$ \\
  \hline
\end{tabular}
\]
\end{center}


\section{Discussion}
\label{sec:discu}
In this paper, we proved that deconvolution of spherical data is possible without any knowledge of the distribution of the noise, and that the radius of the sphere can be recovered at nearly parametric rate. The question whether the rate $1/\sqrt{n}$ can be attained is still open. To get the almost parametric rate following the proposed analysis would require first to be able to strengthen the lower bound of $M$ in \eqref{eq:minoM}. But in \cite{LCGL2021}, getting a lower bound for $M$ requires delicate arguments involving a technical truncation from which it is not possible to derive a quadratic lower bound. If ever such a lower bound can be proved, new ideas have to be developed. 
Also, we were able to prove the identifiability theorem for all possible densities on a circle, but in higher dimensions the proof holds only for densities that are positive near the origin. Extending the result to hold for any density for any $d$ would be nice.
\\
We also proved, for noisy data on a circle that the exploration density can be recovered at the same minimax convergence rate on Sobolev regularity classes as when the noise distribution is known. The analysis we propose here does not extend to $d>2$, and the question of the convergence rate for $d>2$ remains unsolved.\\
More generally, deconvolution of data coming from observations supported on a lower dimensional manifold and corrupted by additive noise has been investigated earlier for known noise in \cite{GPVW12}, see also \cite{Brunel21}. The extension of the methodology proposed here to analyze those settings and to deal with unknown noise distribution will be developed in a further work. Understanding how to deal with noisy observations in topological data analysis is a challenging question,
see for instance \cite{AL19} and \cite{ACC17}, and our solution for additive noise having independent components can be understood as a contribution in this perspective. 

\section{Proofs}
\label{sec:proofs}

\subsection{Proof of Proposition \ref{prop:phihat}}
\label{subsec:phihat}
We shall denote $\|\cdot \|_{2,\nu}$ the ${\mathbb L}^{2}(B_{\nu}^m)$-norm and $\|\cdot \|_{\infty,\nu}$ the ${\mathbb L}^{\infty}(B_{\nu}^m)$-norm, where the dimension $m$ may be $d$, $d_1$ or $d_2$ and is clear from the context.\\
Following the proof of Proposition A.2 in \cite{LCGL2021} and the proof of Proposition 24 in Appendix B.5 in \cite{LCGLSup2021}, we easily get that 
there exist positive constants $b$, $\eta_{1}<1$ and $\eta_{2}<1$  depending only on 
$\nu$, $S$, $d$, $\rho$ such that for all $\Phi \in \Upsilon_{\rho,S}$, 
\begin{equation}
\label{eq:minoM}
\|h\|_{2,\nu}\leq \eta_{1}
\Longrightarrow
M(\Phi+h;\nu \vert \Phi) \geq c_{\nu}^{4}\|h\|_{2,\nu}^{2+2\epsilon (\|h\|_{2,\nu})},
\end{equation}
with, for any $u\in (0,1)$,
$$
\epsilon (u)=\frac{b}{\log \log \frac{1}{u}},
$$
and such that for any  
$t_{1}\in\R^{d_1}$ and any $t_{2}\in\R^{d_2}$,
\begin{eqnarray}
\label{eq:minohi}
\|h\|_{2,\nu}\leq \eta_{2}
&\Longrightarrow &
\|h (\cdot,0)\|_{2,\nu}^{2}\leq \|h\|_{2,\nu}^{2-2\epsilon (\|h\|_{2,\nu})}\nonumber\\
&&{\text{ and }}
\|h (0,\cdot)\|_{2,\nu}^{2}\leq \|h\|_{2,\nu}^{2-2\epsilon (\|h\|_{2,\nu})}.
\end{eqnarray}
We now fix $\eta  = \eta_{1}\wedge \eta_{2}$.
Let  $Z_{n}$ be the random process defined, for all $t=(t_{1},t_{2})\in\R^{d_1} \times \R^{d_2}$, by
$$
Z_{n}(t)=\sqrt{n}\left(\tilde{\phi}_{n}(t)-\Phi(t)\Phi_{\varepsilon^{(1)}}(t_1)\Phi_{\varepsilon^{(2)}}(t_2)  \right)
$$
Using explicit computation, straightforward upper bounds and \eqref{eq:minohi} we easily get that there exists a constant $C$ that depends only on  $\nu_{\text{est}}$, $\rho$ and $S$ such that if $h\in \Upsilon_{\rho,S}$ is such that $\|h\|_{2,\nu_{\text{est}}}\leq \eta $, then
\begin{eqnarray}
\label{eq:majooscil}
\left |\left(M_{n}(\Phi + h)-M(\Phi +h,\nu_{\text{est}}\vert \Phi)\right)\right.
&-&\left.\left(M_{n}(\Phi)-M(\Phi,\nu_{\text{est}}\vert \Phi)\right) \right|\\
& \leq  &C \left[\frac{\|Z_{n}\|_{\infty,\nu_{\text{est}}} }{\sqrt{n}}+\frac{\|Z_{n}\|^{2}_{\infty,\nu_{\text{est}}} }{n}\right]\cdot
\left[\|h\|_{2,\nu_{\text{est}}}^{1-\epsilon (\|h\|_{2,\nu_{\text{est}}})}+ \|h\|_{2,\nu_{\text{est}}}^{2-2\epsilon (\|h\|_{2,\nu_{\text{est}}})}\right]\nonumber\\
&\leq &3 C  \frac{\|Z_{n}\|_{\infty,\nu_{\text{est}}} }{\sqrt{n}}\cdot
\|h\|_{2,\nu_{\text{est}}}^{1-\epsilon (\|h\|_{2,\nu_{\text{est}}})}
\end{eqnarray}
since $\|Z_{n}\|_{\infty,\nu_{\text{est}}}\leq 2\sqrt{n}$ and $\|h\|_{2,\nu_{\text{est}}}\leq 1$.
We shall use the following deviation inequality which is proved  in Appendix G of  \cite{LCGLSup2021}. There exist a numerical constant $c$ and a constant $C$ that depends only on $d$, $\nu_{\text{est}}$ and $E(\|Y_{1}\|^2)$ such that for all $n\geq 1$ and $x>0$, with probability at least $1-4e^{-x}$,
\begin{equation}
\label{eq:deviaZn}
\|Z_{n}\|_{\infty,\nu_{\text{est}}} \leq C+c\sqrt{x}+c\frac{x}{\sqrt{n}}.
\end{equation}
The end of the proof follows from a classical peeling device. Let 
$\widehat{h}=\widehat{\phi}_{n}-\Phi$.
For any $a>0$, let $j_{0}$ the integer such that $2^{j_0}\leq a < 2^{j_0 +1}$, and $J_{n}(a)=\{j\in\N\;:j_{0}\leq j,\;2^{j  +1}\leq r_{n}\eta \}$. Then
$$
P\left(r_{n}\|\widehat{h}\|_{2,\nu}\geq a\right)
\leq P\left(\|\widehat{h}\|_{2,\nu}\geq \eta /2\right)+\sum_{j\in J_{n}(a)}
P\left(2^{j}\leq r_{n}\|\widehat{h}\|_{2,\nu}\leq 2^{j  +1}\right).
$$
Now, if $2^{j}\leq r_{n}\|\widehat{h}\|_{2,\nu}\leq 2^{j  +1}$, then
$$
\inf_{h:2^{j}\leq r_{n}\|h\|_{2,\nu}\leq 2^{j  +1}}M_{n}(\Phi + h)
\leq M_{n} (\Phi),
$$
which implies that
\begin{eqnarray*}
\sup_{h:2^{j}\leq r_{n}\|h\|_{2,\nu}\leq 2^{j  +1}}
\left |\left(M_{n}(\Phi + h)-M(\Phi +h,\nu_{\text{est}}\vert \Phi)\right)\right.
&-&\left.\left(M_{n}(\Phi)-M(\Phi,\nu_{\text{est}}\vert \Phi)\right) \right|\\
&\geq &  \inf_{h:2^{j}\leq r_{n}\|h\|_{2,\nu}\leq 2^{j  +1}, \|h\|_{2,\nu}\leq \eta (\delta)}M(\Phi+h,\nu_{\text{est}}\vert \Phi)\\
&\geq &  \inf_{h:2^{j}\leq r_{n}\|h\|_{2,\nu}\leq 2^{j  +1}, \|h\|_{2,\nu}\leq \eta (\delta)}M(\Phi+h,\nu\vert \Phi)
\end{eqnarray*}
where the last inequality follows from the fact that since $\nu\leq \nu_{\text{est}}$, 
$M(\cdot,\nu_{\text{est}}\vert \Phi)\geq M(\cdot,\nu\vert \Phi)$.
Using \eqref{eq:minoM} and \eqref{eq:majooscil} we then get that for $j\in J_{n}(a)$,
$$P\left(2^{j}\leq r_{n}\|\widehat{h}\|_{2,\nu}\leq 2^{j  +1}\right)
\leq
P\left(c_{\nu}^4 \left(\frac{2^{j}}{r_{n}}\right)^{2+2\epsilon(2^{j}/r_{n})}\leq 3C\frac{\|Z_{n}\|_{\infty,\nu_{\text{est}}} }{\sqrt{n}}\left(\frac{2^{j  +1}}{r_{n}}\right)^{1-\epsilon(2^{j+1}/r_{n})}
\right)
$$
so that
$$P\left(2^{j}\leq r_{n}\|\widehat{h}\|_{2,\nu}\leq 2^{j  +1}\right)
\leq
P\left(\|Z_{n}\|_{\infty,\nu_{\text{est}}} \geq \frac{c_{\nu}^4}{3C}2^{j} \sqrt{n}\left(\frac{1}{r_{n}}\right)^{1+3\epsilon(2^{j}/r_{n})}\right).
$$
Now, for $j\in J_{n}(a)$, $(2^{j}/r_{n}) \leq \eta$, and for $c(\eta)=3/\log \log (1/\eta )$ we have $3\epsilon(2^{j}/r_{n})\leq c(\eta)$. We then get for $j\in J_{n}(a)$, 
$$P\left(2^{j}\leq r_{n}\|\widehat{h}\|_{2,\nu}\leq 2^{j  +1}\right)
\leq
P \left(\|Z_{n}\|_{\infty,\nu_{\text{est}}} \geq \frac{c_{\nu}^4}{3C}2^{j} \sqrt{n}\left(\frac{1}{r_{n}}\right)^{1+c(\eta)}\right).
$$
We now take  $r_n$ such that $r_{n}^{1+c(\eta)}= \sqrt{n}$, that is 
$r_{n}=n^{1/2(1+c(\eta))}$, and we get
\begin{equation}
\label{eq:majoproba}
P\left(r_{n}\|\widehat{h}\|_{2,\nu}\geq a\right) \leq P\left(\|\widehat{h}\|_{2,\nu}\geq \eta /2\right)+\sum_{j\geq j_{0}}P\left(\|Z_{n}\|_{\infty,\nu_{\text{est}}} \geq C2^{j}
\right)
\end{equation}
for some constant $C$. 
To deal with the first term in the upper bound we use (25) in \cite{LCGL2021}.
The second term is a series  which is upper bounded using \eqref{eq:deviaZn} and Proposition \ref{prop:phihat} follows.

\subsection{Proof of Lemma \ref{lem:condi}}
\label{subsec:condi}
To begin with, for any $z_{0}\in\C$ and $z\in\C^{d-1}$,
$$
E\left[\exp\left(iz_{0}\widetilde{X}^{(1)}+iz^{T}\widetilde{X}^{(2)}\right)\right]=
E\left[E\left[\exp\left(iz_{0}\widetilde{X}^{(1)}\right)\vert \widetilde{X}^{(2)}\right]  \exp\left(iz^{T}\widetilde{X}^{(2)}\right)\right],
$$
and usual arguments for multivariate analytic functions on $\C^{d-1}$ prove that $z\mapsto E\left[\exp\left(iz_{0}\widetilde{X}^{(1)}+iz^{T}\widetilde{X}^{(2)}\right)\right]$ is the nul function if and only if $E\left[\exp\left(iz_{0}\widetilde{X}^{(1)}\right)\vert \widetilde{X}^{(2)}\right]$ is zero $P_{\widetilde{X}^{(2)}}$-a.s. In the same way,
for any  $z_{0}\in\C^{d-1}$, $z\mapsto E\left[\exp\left(iz\widetilde{X}^{(1)}+iz_{0}^{T}\widetilde{X}^{(2)}\right)\right]$  is the nul function if and only if $E\left[\exp\left(iz_{0}^T\widetilde{X}^{(2)}\right)\vert \widetilde{X}^{(1)}\right]$ is zero $P_{\widetilde{X}^{(1)}}$-a.s.
Also, the value of the center $C^{\star}$ can only change the function $E\left[\exp\left(z^T(\widetilde{X}^{(1)},\widetilde{X}^{(2)})\right)\right]$, $z\in \C^{d}$, by a factor $\exp(z^T C^{\star})$ which is non zero, so that 
we may assume $C^{\star}=0$ to prove Lemma \ref{lem:condi}.
\\
In the following, we write for all $u \in (0,1), \widetilde{S}^{(1)}(u) = cos(2 \pi u)$ and for all $u \in (0,1)^{d-1}$ , 
$$\widetilde{S}^{(2)}(u) = \begin{pmatrix} \sin(2 \pi u^{(1)}) \cos(\pi u^{(2)}) \\ \sin(2 \pi u^{(1)}) \sin( \pi u^{(2)}) \cos( \pi u^{(3)}) \\ \vdots \\ \sin(2 \pi u^{(1)}) \sin(\pi u^{(2)}) \cdots \sin(\pi u^{(d-2)}) \cos(\pi u^{(d-1)}) \\ \sin(2 \pi u^{(1)}) \sin(\pi u^{(2)}) \cdots \sin(\pi u^{(d-2)}) \sin(\pi u^{(d-1)})  \end{pmatrix}.$$
%
%
We first prove that for any $z_0 \in \mathbb{C}$, $E\left[\exp\left(iz_{0}\widetilde{X}^{(1)}\right)\vert \widetilde{X}^{(2)}\right]$ is not $P_{\widetilde{X}^{(2)}}$-a.s.  zero.
\\
%
%
Since $f^{\star}$ is not identically zero, there exists a closed interval $[\alpha, \beta]$ subset of one of the four following intervals : $(0,\frac{1}{4})$, $(\frac{1}{4}, \frac{1}{2})$, $(\frac{1}{2},\frac{3}{4})$, $(\frac{3}{4},1)$, a vector $a = (a^{(i)})_{1 \leq i \leq d-1} \in (\alpha, \beta) \times (0,1)^{d-2}$ (if $d=2$, $a$ is a real number in  $ (\alpha, \beta)$) and $\lambda > 0$ such that if we define 
$$I_1 = \lbrace u=(u^{(i)})_{i \in \lbrace 1, \ldots, d-1 \rbrace} \in (\alpha, \beta) \times (0,1)^{d-2} \; :  \; \left\| \widetilde{S}^{(2)}(u) - \widetilde{S}^{(2)}(a)\right\|^2_2 < \lambda^2 \rbrace,
$$ 
then the restriction of $f^{\star}$ to $I_{1}$, 
$f^{\star} \vert_{I_1}$, is not the null function.\\
We choose $\lambda$ small enough such that, if we define $A \subset (-1,1)^{d-1}$ as $A = \widetilde{S}^{(2)}(I_1)$, we have that there exists $I_2 \subset (0,1)^{d-1}$ such that $I_{1} \cap I_{2}=\emptyset$ and $(\widetilde{S}^{(2)})^{-1} (A) = I_1 \cup I_2$.\\
We define, for $i,j \in \lbrace 1,2 \rbrace$, the $\mathcal{C}^1$ diffeomorphisms $\eta_{i,j} : I_i \longrightarrow I_j$, $ u \mapsto (S^{(2)})^{-1}(S^{(2)}(u))$, such that $\eta_{i,i} = Id\vert_{I_i}$ and $\eta_{i,j} \circ \eta_{j,i} = Id \vert_{I_i}$.\\
Note that we can explicitly calculate $\eta_{i,j}(x)$ for the different possible inclusions of $[\alpha, \beta]$:
\begin{enumerate}
\item for
 $[\alpha, \beta] \subset (0,\frac{1}{4})$ or $[\alpha, \beta] \subset (\frac{1}{4},\frac{1}{2})$, we have $\eta_ {1,2}(u^{(1)},u^{(2)}, \ldots, u^{(d-1)}) = (\frac{1}{2}-u^{(1)}, u^{(2)}, \ldots, u^{(d-1)})$, 
 \item
and for $[\alpha, \beta] \subset (\frac{1}{2},\frac{3}{4})$ or $[\alpha, \beta] \subset (\frac{3}{4},1)$ , we have $\eta_{1,2}(u^{(1)}, u^{(2)}, \ldots, u^{(d-1)}) = (\frac{3}{2}-u^{(1)}, u^{(2)}, \ldots, u^{(d-1)})$. 
\end{enumerate}
 %
We now compute $E[\exp( iz_0 \widetilde{X}^{(1)} ) \vert \widetilde{X}^{(2)} ] 1_{\widetilde{X}^{(2)} \in A}$.
For any measurable bounded function 
$\omega$  on $(-R^{\star},R^{\star})^{d-1}$, we have
\begin{eqnarray*}
E[\exp(i z_0 \widetilde{X}^{(1)} ) \omega(\widetilde{X}^{(2)}) 1_{\widetilde{X}^{(2)}\in A}]
&=&
 \int_{(\widetilde{S}^{(2)})^{-1}(A)} \omega(R^{\star}\widetilde{S}^{(2)}(u)) \exp(iz_0 R^{\star} \cos (2 \pi u^{(1)})) f^{\star}(u) du\\
&=&\int_{I_1} \omega(R^{\star}\widetilde{S}^{(2)}(u)) \exp(i z_0 R^{\star} \cos(2 \pi u^{(1)})) f^{\star}(u) du \\
&&+ \int_{I_2} \omega(R^{\star}\widetilde{S}^{(2)}(u)) \exp(i z_0 R^{\star} \cos(2 \pi u^{(1)})) f^{\star}(u) du.
 \end{eqnarray*}
Define the  change of variables $u = \eta_{1,2}(v)$ in the second integral. Using the explicit definition of $\eta_{1,2}$ which is differentiable with Jacobian equal to $1$ we get
\begin{eqnarray*}
&&E[\exp(i z_0 \widetilde{X}^{(1)} ) \omega(\widetilde{X}^{(2)}) 1_{\widetilde{X}^{(2)}\in A}]\\
&&=
\int_{I_1} \omega(R^{\star}\widetilde{S}^{(2)}(u)) \lbrace \exp(i R^{\star} z_0 \cos(2 \pi u^{(1)})) f^{\star}(u) + \exp(i R^{\star} z_0 \cos(2 \pi\eta_{1,2}(u)^{(1)})) f^{\star}(\eta_{1,2}(u)) \rbrace du\\
&&= \int_{I_1} \omega(R^{\star}\widetilde{S}^{(2)}(u)) \frac{\exp(i R^{\star} z_0 \cos(2 \pi u^{(1)})) f^{\star}(u) + \exp(i R^{\star} z_0 \cos(2 \pi\eta_{1,2}(u)^{(1)})) f^{\star}(\eta_{1,2}(u)) }{f^{\star}(u) + f^{\star}(\eta_{1,2}(u))} f^{\star}(u) du \\
&& +\int_{I_1} \omega(R^{\star}\widetilde{S}^{(2)}(u)) \frac{ \exp(i R^{\star} z_0 \cos(2 \pi u^{(1)})) f^{\star}(u) + \exp(i R^{\star} z_0 \cos(2 \pi\eta_{1,2}(u)^{(1)})) f^{\star}(\eta_{1,2}(u)) }{f^{\star}(u) + f^{\star}(\eta_{1,2}(u))} f^{\star}(\eta_{1,2}(u)) du.
 \end{eqnarray*}
Thus if we define $\nu_{1}:A \longrightarrow I_{1}$ such that for all $u\in I_{1}$, $\nu_{1}(\widetilde{S}^{(2)}(u))=u$, and $\nu_{2}:A \longrightarrow I_{2}$ such that for all $u\in I_{2}$, $\nu_{2}(\widetilde{S}^{(2)}(u))=u$, we get
\begin{eqnarray*}
&&E\left[\exp(i z_0 \widetilde{X}^{(1)} ) \vert \widetilde{X}^{(2)}\right ] 1_{\widetilde{X}^{(2)} \in A}\\
&&= \frac{ \exp(i R^{\star} z_0 \cos(2 \pi\nu_{1}(\widetilde{X}^{(2)}))^{(1)}) f^{\star}(\nu_{1}(\widetilde{X}^{(2)})) + \exp(-i R^{\star} z_0 \cos(2 \pi\nu_{2}(\widetilde{X}^{(2)}))^{(1)}) f^{\star}(\nu_{2}(\widetilde{X}^{(2)})) }{f^{\star}(\nu_{1}(\widetilde{X}^{(2)})) + f^{\star}(\nu_{2}(\widetilde{X}^{(2)}))} 1_{\widetilde{X}^{(2)} \in A}.
\end{eqnarray*}
Finally, $E\left[\exp(i z_0 \widetilde{X}^{(1)} ) \vert \widetilde{X}^{(2)}\right ] 1_{\widetilde{X}^{(2)} \in A}$ is nul $\mathbb{P}_{\widetilde{X}^{(2)}}$-a.s if and only if  for $f^{\star}du$ almost all $u \in I_1$,
$$ \exp( i R^{\star} z_0 \cos(2 \pi u^{(1)})) f^{\star}(u) + \exp( i R^{\star} z_0 \cos(2 \pi \eta_{1,2}(u)^{(1)})) f^{\star}(\eta_{1,2}(u)) = 0,$$
that is  for $f^{\star}du$ almost all $u \in I_1$,
$$
 \exp\left( i R^{\star} z_0  \left(\cos(2 \pi u^{(1)})- \cos(2 \pi\eta_{1,2}(u)^{(1)}) \right)\right)=- \frac{f^{\star}(\eta_{1,2}(u))}{f^{\star}(u)}.
$$
Since for almost all $u \in I_1$, $f^{\star}(u) \neq 0$, this would imply in particular that for almost all $u \in I_1$

\begin{equation}
 R^{\star}  \text{Re}(z_0)  \left(\cos(2 \pi u^{(1)}) - \cos(2 \pi\eta_{1,2}(u)^{(1)}) \right) = \pi\;\; [\text{mod} 2 \pi],
\label{eq:X1null}
\end{equation}
which gives a contradiction.
\\

Now, let's prove that for any $z_0 \in \mathbb{C}^{d-1}$, $E\left[\exp\left(iz_{0}^T\widetilde{X}^{(2)}\right)\vert \widetilde{X}^{(1)}\right]$ is not $P_{\widetilde{X}^{(1)}}$-a.s. zero.\\
Since $f^{\star}$ is not identically zero, there exists a closed interval $[\alpha, \beta]$ in one of the four following intervals : $(0,\frac{1}{4})$, $(\frac{1}{4}, \frac{1}{2})$, $(\frac{1}{2},\frac{3}{4})$, $(\frac{3}{4},1)$, such that $f|_{(\alpha, \beta) \times (0,1)^{d-2}}$ is not the null function. \\
Define $J_1 = (\alpha, \beta)$, $B = \widetilde{S}^{(1)}(J_1)$, and $J_2$ such that $J_2 \cap J_1 =\emptyset$ and $(\widetilde{S}^{(1)})^{-1}(B) = J_1 \cup J_2$. We define the $\mathcal{C}^1$ diffeomorphism $\sigma_{1,2} : J_1 \longrightarrow J_2$, $ u \mapsto (\widetilde{S}^{(1)})^{-1}(\widetilde{S}^{(1)}(u))$, such that $\sigma_{1,1} = Id|_{I_1}$ and $\sigma_{1,2} \circ \sigma_{2,1} = Id|_{I_1}$. Note that we can explicitly calculate $\sigma_{1,2}(u)$, indeed, for $u \in J_1$, we have $\sigma_{1,2}(u) = 1-u $. The reason of choosing $J_1$ in one of these four intervals is to have the decomposition of $(\widetilde{S}^{(1)})^{-1}(B)$ in exactly 2 disjoint open sets on which $\widetilde{S}^{(1)}$ is one to one.
 \\
For any bounded and measurable function $\omega$ on $(-R^{\star},R^{\star})$, we have
\begin{eqnarray*}
E\left[\exp\left(iz_{0}^T\widetilde{X}^{(2)}\right) \omega(\widetilde{X}^{(1)}) 1_{\widetilde{X}^{(1)} \in B}\right] 
& =&\int_{J_1 \times (0,1)^{d-2}} \omega(R^{\star}\cos(2 \pi u^{(1)})) \exp(iR^{\star} z_0^{\intercal} \widetilde{S}^{(2)}(u)) f^{\star}(u) du\\
&& + \int_{J_2 \times (0,1)^{d-2}} \omega(R^{\star}\cos(2 \pi u^{(1)})) \exp(i R^{\star}z_0^{\intercal} \widetilde{S}^{(2)}(u)) f^{\star}(u) du.
\end{eqnarray*}
Define the following change of variables for the second integral, $u = \overline{\sigma}(v)=(\sigma_{1,2} \times Id_{{(0,1)}^{d-2}})(v)$, which has Jacobien equal to 1. Then
\begin{eqnarray*}
&&E\left[\exp\left(iz_{0}^T\widetilde{X}^{(2)}\right) \omega(\widetilde{X}^{(1)}) 1_{\widetilde{X}^{(1)} \in B}\right]\\
&& = \int_{J_1 \times (0,1)^{d-2}}  \omega(R^{\star}\cos(2 \pi u^{(1)})) \lbrace \exp(i R^{\star}z_0^{\intercal} \widetilde{S}^{(2)}(u)) f^{\star}(u)
+ \exp(iR^{\star} z_0^{\intercal} \widetilde{S}^{(2)}(\overline{\sigma}(u))) f^{\star}(\overline{\sigma}(u)) \rbrace du\\
&&= \int_{J_1 \times (0,1)^{d-2}}  \omega(R^{\star}\cos(2 \pi u^{(1)})) \frac{   \exp(i R^{\star} z_0^{\intercal} \widetilde{S}^{(2)}(u)) f^{\star}(u)
+ \exp(iR^{\star} z_0^{\intercal} \widetilde{S}^{(2)}(\overline{\sigma}(u))) f^{\star}(\overline{\sigma}(u)) }{f^{\star}(u) + f^{\star}(\overline{\sigma}(u))}f^{\star}(u)du\\
&& + \int_{J_1 \times (0,1)^{d-2}}  \omega(R^{\star}\cos(2 \pi u^{(1)})) \frac{ \exp(i R^{\star} z_0^{\intercal} \widetilde{S}^{(2)}(u)) f^{\star}(u)
+ \exp(i R^{\star} z_0^{\intercal} \widetilde{S}^{(2)}(\overline{\sigma}(u))) f^{\star}(\overline{\sigma}(u)) }{f^{\star}(u) + f^{\star}(\overline{\sigma}(u))}f^{\star}(\overline{\sigma}(u))du.
\end{eqnarray*}
We now define $\tau_{1}:B \longrightarrow J_{1}$ such that for all $u\in J_{1}$, $\tau_{1}(\widetilde{S}^{(1)}(u))=u$, and $\tau_{2}:B \longrightarrow J_{2}$ such that for all $u\in J_{2}$, $\tau_{2}(\widetilde{S}^{(1)}(u))=u$.\\
First when $d=2$, we get
\begin{eqnarray*}
&&E\left[\exp\left(iz_{0}^T\widetilde{X}^{(2)}\right) \vert \widetilde{X}^{(1)} \right]1_{\widetilde{X}^{(1)} \in B}\\
&&=\frac{ \exp(iR^{\star} z_0^{\intercal} \widetilde{S}^{(2)}({\tau}_{1}(\frac{\widetilde{X}^{(1)}}{R^{\star}}))) f^{\star}({\tau}_{1}(\frac{\widetilde{X}^{(1)}}{R^{\star}}))
+ \exp(-i R^{\star} z_0^{\intercal} \widetilde{S}^{(2)}({\tau}_{2}(\frac{\widetilde{X}^{(1)}}{R^{\star}}))) f^{\star}({\tau}_{2}(\frac{\widetilde{X}^{(1)}}{R^{\star}})) }
{f^{\star}({\tau}_{1}(\frac{\widetilde{X}^{(1)}}{R^{\star}}))+f^{\star}({\tau}_{2}(\frac{\widetilde{X}^{(1)}}{R^{\star}}))}1_{\widetilde{X}^{(1)} \in B}
\end{eqnarray*}
and $E\left[\exp\left(iz_{0}^T\widetilde{X}^{(2)}\right) \vert \widetilde{X}^{(1)} \right]$ can not be nul $\mathbb{P}_{\widetilde{X}^{(1)}}$-a.s. since it would require that for almost all $(u,v)\in J_1 \times (0,1)^{d-2}$,
$$
\exp(iR^{\star} z_0^{\intercal} \widetilde{S}^{(2)}(u,v)) f^{\star}(u,v)
+ \exp(i R^{\star} z_0^{\intercal} \widetilde{S}^{(2)}(1-u,v)) f^{\star}(1-u,v) =0.
$$
Then when $d>2$, we can choose $J_{1}=(0,\gamma]$
\begin{eqnarray*}
&&E\left[\exp\left(iz_{0}^T\widetilde{X}^{(2)}\right) \vert \widetilde{X}^{(1)} \right]1_{\widetilde{X}^{(1)} \in B}\\
&&=\left(\int_{(0,1)^{d-2}}\frac{ \exp(iR^{\star} z_0^{\intercal} \widetilde{S}^{(2)}({\tau}_{1}(\frac{\widetilde{X}^{(1)}}{R^{\star}}),v)) f^{\star}({\tau}_{1}(\frac{\widetilde{X}^{(1)}}{R^{\star}}),u)
+ \exp(-i R^{\star} z_0^{\intercal} \widetilde{S}^{(2)}({\tau}_{2}(\frac{\widetilde{X}^{(1)}}{R^{\star}}),v)) f^{\star}({\tau}_{2}(\frac{\widetilde{X}^{(1)}}{R^{\star}}),u) }
{f^{\star}({\tau}_{1}(\frac{\widetilde{X}^{(1)}}{R^{\star}}),u)+f^{\star}({\tau}_{2}(\frac{\widetilde{X}^{(1)}}{R^{\star}}),u)}du\right)1_{\widetilde{X}^{(1)} \in B}
\end{eqnarray*}
and $E\left[\exp\left(iz_{0}^T\widetilde{X}^{(2)}\right) \vert \widetilde{X}^{(1)} \right]1_{\widetilde{X}^{(1)} \in B}$ is nul $\mathbb{P}_{\widetilde{X}^{(1)}}$-a.s. if and only if for all $u\in J_{1}$,
$$
\int_{(0,1)^{d-2}}\frac{ \exp(iR^{\star} z_0^{\intercal} \widetilde{S}^{(2)}(u,v)) f^{\star}(u,v)
+ \exp(i R^{\star} z_0^{\intercal} \widetilde{S}^{(2)}(1-u,v)) f^{\star}(1-u,v) }
{f^{\star}(u,v)+f^{\star}(1-u,v)}dv =0.
$$
In particular, this implies that for all $u\in J_{1}$,
$$
\int_{(0,1)^{d-2}}
\cos({\text{Re}}(z_0)^{\intercal} \widetilde{S}^{(2)}(u,v)) \frac{ [f^{\star}(u,v) \exp(-R^{\star} {\text{Im}}(z_0)^{\intercal} \widetilde{S}^{(2)}(u,v))
+f^{\star}(1-u,v) \exp(R^{\star} {\text{Im}}(z_0)^{\intercal} \widetilde{S}^{(2)}(u,v))] }
{f^{\star}(u,v)+f^{\star}(1-u,v)}dv =0.
$$
But for small enough $u\in J_{1}$, $\cos({\text{Re}}(z_0)^{\intercal} \widetilde{S}^{(2)}(u,v))$ stays positive for all $v\in (0,1)^{d-2}$ which gives a contradiction.

\subsection{Proof of Proposition \ref{prop:consi}}
\label{subsec:consi}

For $f,f^{\star} \in \mathcal{F}$ and $R,R^{\star} \in [R_{\text{min}},R_{\text{max}}]$, denote $\theta = (f,R)$, $\theta^{\star} = (f^{\star},R^{\star})$, and define the distance $d$ by  $d(\theta, \theta^{\star}) = \left(\int_{(0,1)^{d-1}}(f(u)-f^{\star}(u))^2 du\right)^{1/2}  + |R - R^{\star}|$. \\
First, using Lemma A.1 in \cite{LCGL2021} we get
\begin{equation}
\label{eq:consi:1}
\sup_{\theta \in \mathcal{F}\times [R_{\text{min}}, R_{\text{max}}]}\left| M_n(\theta) - M(\theta)\right| = o_{\mathbb{P}_{C^{\star},R^{\star},f^{\star},\mathbb{Q}^{\star}}}(1).
\end{equation}
Then, using the continuity of $M$ with respect to the distance $d$ and the compacity of $\mathcal{F}\times [R_{\text{min}}, R_{\text{max}}]$, using Theorem \ref{theo:ident} we get that for any $\delta > 0$, 
\begin{equation}
\label{eq:consi:2}
\inf_{ \theta \in \mathcal{F} \times [R_{\text{min}}, R_{\text{max}}] , \  d(\theta,\theta^{\star}) > \delta}M(\theta) > M(\theta^{\star}) = 0.
\end{equation}
Consistency of $\widehat{f}$ and $\widehat{R}$ follows from \eqref{eq:consi:1}, \eqref{eq:consi:2} 
and Theorem 5.7
in \cite{vdV98}. Consistency of $\widehat{C}$ is then a consequence of the continuity theorem and the law of large numbers.

\subsection{Proof of proposition \ref{prop:density}}
\label{subsec:propdensity}

The functions $\Psi_{f,R}$ on $\R^2$ can be written as functions $\Phi_{f,R}$ on $[0,+\infty[\times  [0,1)$ using polar representation. For any $r\geq 0$ and $\theta\in [0,1)$, define
$$
\Phi_{f,R}(r,\theta)=\Psi_{f,R}(r \cos (2 \pi \theta), r\sin (2 \pi \theta)).
$$
For all $r \geq 0$, let $(\lambda_p(r))_{p \in \mathbb{Z}}$ be the sequence of Fourier coefficients of $\Phi_{f,R}(r,\cdot)$, 
$$\lambda_p(r) = \int_0^{1} \Phi_{f,R}(r,u) e^{2 i \pi p u} du, \ \ p \in \mathbb{Z}.$$
Using (III) in Section \ref{sec:Bessel}, we get that for all $r \geq 0$,
\begin{eqnarray*} 
\Phi_{f,R}(r,\theta)& = &\int_0^{1} f (u) \text{exp}(i r R  cos(2 \pi u - 2 \pi \theta))du\\
& = &\sum_{p \in \mathbb{Z}} \sum_{k  \in \mathbb{Z}} i^p J_p(r R ) f_k (\int_{0}^{1}  e^{-2 i \pi u(k-p)}du)e^{- 2 i \pi p \theta}\\
&= &\sum_{p \in \mathbb{Z}} i^p f_p J_p(r R) e^{- 2 i \pi p \theta},
\end{eqnarray*}
so that for all $p \in \mathbb{Z}$,
$$ \lambda_p(r) = i^p J_p(r R ) f_p,$$
and also
$$\lambda_p^{\star}(r) = i^p J_p(r R^{\star} )f_p^{\star}, \ \ \ \widehat{\lambda}_p(r) = i^p J_p(r \widehat{R}) \widehat{f}_p,$$
where $(f_p^{\star})_{p \in \mathbb{Z}}$ (resp. $(\widehat{f}_p)_{p \in \mathbb{Z}}$) are the Fourier coefficients of $f^{\star}$(resp. $\hat{f}$) and $(\widehat{\lambda}_p(r))_{p \in \mathbb{Z}}$ are the Fourier coefficients of $\Phi_{\widehat{f},\widehat{R}}(r,.)$. 
We have, using Parseval's identity,
\begin{eqnarray*} 
\int_{0}^{1} |\Phi_{f^{\star},R^{\star}}(r,\theta) - \Phi_{\widehat{f},\widehat{R}}(r,\theta)|^2 d\theta &=&  \sum_{k \in \mathbb{Z}} |\lambda_k^{\star}(r) - \widehat{\lambda}_k(r)|^2 \\
& = &\sum_{k \in \mathbb{Z}} |f_k^{\star} J_k(rR^{\star}) - \widehat{f}_k J_k(r\widehat{R})|^2 \\
&\geq & \sum_{|k| \leq N} |f_k^{\star} J_k(rR^{\star}) - \widehat{f}_k J_k(r \widehat{R})|^2. 
\end{eqnarray*}
We use the fact that $ |a - b|^2 \geq \frac{|a|^2}{2} - |b|^2$ for all $a,b \in \mathbb{C}$, to get
$$ \int_{0}^{1} |\Phi_{f^{\star},R^{\star}}(r,\theta) - \Phi_{\widehat{f},\widehat{R}}(r,\theta)|^2 d\theta \geq  \sum_{|k| \leq N} \frac{|f_k^{\star} - \widehat{f}_k|^2}{2} J_k(r\widehat{R})^2 - \sum_{|k| \leq N} |f^{\star}_k|^2 |J_k(rR^{\star})-J_k(r\widehat{R})|^2,$$
so that
$$ \sum_{|k| \leq N} |f_k^{\star} - \widehat{f}_k|^2 J_k(r\widehat{R})^2 \leq 2 \int_{0}^{1} |\Phi_{f^{\star},R^{\star}}(r,\theta) - \Phi_{\widehat{f},\widehat{R}}(r,\theta)|^2 d\theta + \sum_{|k| \leq N} |f_k^{\star}|^2 |J_k(rR^{\star})-J_k(r\widehat{R})|^2.$$
Then, 
for all $\nu\in(0,\nu_\text{est}]$ such that $c_{\nu}^{\star}>0$, we integrate from $0$ to $\nu$ and we use (IV) in Section \ref{sec:Bessel} to obtain

$$ \sum_{|k| \leq N} |f_k^{\star} - \widehat{f}_k|^2 \int_{0}^{\nu} r J_k(r\widehat{R})^2 dr \leq 2 ||\Psi_{f^{\star},R^{\star}} - \Psi_{\widehat{f},\widehat{R}}||^2_{\mathbb{L}_2(\mathbb{D}_2(0,\nu))} + 2 \sum_{|k| \leq N} |f_k^{\star}|^2 |R^{\star}-\widehat{R}|^2 \int_{0}^{\nu} r^2 dr.$$
Using Proposition \ref{prop:phihat}, Theorem \ref{theo:radius} and the fact that $ \sum_{|k| \leq N} |f_k^{\star}|^2 \leq \int_{0}^{1}(f^{\star}(u))^2 du  $ is uniformly upper bounded in the compact set $\cal F$, we have that there exists a constant $c > 0$ depending on $\delta$, $\nu$, $c_{\nu}^{\star}$, $d$, $R^{\star}$, $R_{\text{min}}$, $R_{\text{max}}$, and $E(\|Y\|^2)$ such that for all $x \geq 1$ and for $c_1$ and $c_2$ coming from Proposition \ref{prop:phihat}, for all $n\geq (1\vee xc_{1})/c_{2}$,
with probability at least $1 - e^{-x}$,
$$\sum_{|k| \leq N} |f_k - \widehat{f}_k|^2 \int_{0}^{\nu} r J_k(r\widehat{R})^2 dr \leq c \left(\frac{x}{n^{1-\delta}} \vee \frac{x^2}{n^{2-2\delta}} \right).$$
Using lemma \ref{lemmabessel}, we finally have that with probability at least $1-e^{-x}$,
$$ \sum_{|k| \leq N} |f_k^{\star} - \widehat{f}_k|^2 \leq c \frac{32}{9 \nu^2} (\nu \widehat{R})^{-2N} (N+1) 2^{2N} (N!)^2 \left(\frac{x}{n^{1-\delta}} \vee \frac{x^2}{n^{2-2\delta}} \right).$$
We finally use $\widehat{R} \in [R_{\text{min}}, R_{\text{max}}]$ to end the proof.

\subsection{Proof of lemma \ref{lem:convgauss}}
\label{subsec:convgauss}
The proof of (1) follows from the same arguments as in the proof of Proposition \ref{prop:consi}.
\\
For all $(t_1,t_2) \in \R\times \R^{d-1}$, define
$$ m_{n,R}(t_1,t_2) = \Psi_{f^{\star},R}(t_1,t_2)\tilde{\psi}_n(t_1,0) \tilde{\psi}_n(0,t_2) - \tilde{\psi}_n(t_1,t_2) \Psi_{f^{\star},R}(t_1,0) \Psi_{f^{\star},R}(0,t_2)$$
and
$$m_{R}(t_1,t_2) = \Psi_{f^{\star},R}(t_1,t_2) \Psi_{f^{\star},R^{\star}}(t_1 ,0) \Psi_{f^{\star},R^{\star}}(0 , t_2) - \Psi_{f^{\star},R^{\star}}(t_1,t_2) \Psi_{f^{\star},R}(t_1 ,0) \Psi_{f^{\star},R}(0,t_2),$$
such that $M_n(R) = \int_{B_{\nu_{\text{est}}}\times B_{\nu_{\text{est}}}^{d-1}} |m_{n,R}(t_1,t_2)|^2 dt_1 dt_2$ and $M(R) = \int_{B_{\nu}\times B_{\nu}^{d-1}} |m_{R}(t_1,t_2)|^2 |\Phi_{\epsilon}(t_1, t_2)|^2 dt_1 dt_2$. 

Let us prove (2). Differentiation of $M_n$ gives
$$  M'_n(R) = \int_{B_{\nu_{\text{est}}}\times B_{\nu_{\text{est}}}^{d-1}} \left(\frac{d}{dR} \lbrace m_{n,R}(t_1,t_2) \rbrace \overline{m_{n,R}(t_1,t_2)} + \frac{d}{dR} \lbrace \overline{m_{n,R}(t_1, t_2)} \rbrace m_{n,R}(t_1,t_2) \right) dt_1 dt_2, $$
where $\overline{z}$ denotes the complex conjugate of $z$. Since $\overline{m_{n,R}(t_1,t_2)} = m_{n,R}(-t_1,-t_2)$ we get
$$M'_n(R) = \int_{B_{\nu_{\text{est}}}\times B_{\nu_{\text{est}}}^{d-1}} \left(\frac{d}{dR} \lbrace m_{n,R}(t_1,t_2) \rbrace m_{n,R}(-t_1,-t_2) + \frac{d}{dR} \lbrace m_{n,R}(-t_1,-t_2) \rbrace m_{n,R}(t_1,t_2) \right) dt_1 dt_2.$$
Let $\textbf{Z}_n$ be the random process defined for $(t_1,t_2) \in \mathbb{R} \times \mathbb{R}^{d-1}$ by 
\begin{equation}
\textbf{Z}_n(t_1,t_2) = \sqrt{n} \left ( \tilde{\psi}_n(t_1,t_2) - \Psi_{f^{\star},R^{\star},}(t_1,t_2)  \Phi_{\epsilon}(t_1,t_2) \right ).
\label{proccessZ}
\end{equation}
The random process 
$\textbf{Z}_n$ converges weakly to a Gaussian process $(\textbf{Z}(t_1,t_2))_{(t_1,t_2)}$ in the set of complex continuous functions endowed with the uniform norm.\\
Using \eqref{proccessZ},  $\tilde{\psi}_n(t_1,t_2) = \frac{1}{\sqrt{n}} \textbf{Z}_n(t_1,t_2) + \Psi_{f^{\star},R^{\star}}(t_1,t_2) \Phi_{\epsilon}(t_1,t_2)$, so that
\begin{eqnarray*}
 \sqrt{n} M'_n(R^{\star}) &= &\int_{B_{\nu_{\text{est}}}\times B_{\nu_{\text{est}}}^{d-1}} C(t_1,t_2) \left\{ 
 \Psi_{f^{\star},R^{\star}}(-t_1,-t_2) \left[ \textbf{Z}_n(-t_1,0) \Psi_{f^{\star},R^{\star}}(0, -t_2) + \textbf{Z}_n(0, -t_2) \Psi_{f^{\star},R^{\star}}(- t_1, 0) \right]\right.\\
 &&\left.-
 \textbf{Z}_n(-t_1,-t_2) \Psi_{f^{\star},R^{\star}}(- t_1,0) \Psi_{f^{\star},R^{\star}}(0,- t_2) \right\} dt_{1}dt_{2}
\\
&& + \int_{B_{\nu_{\text{est}}}\times B_{\nu_{\text{est}}}^{d-1}} C(-t_1, -t_2)  \left\{ 
 \Psi_{f^{\star},R^{\star}}(t_1,t_2) \left[ \textbf{Z}_n(t_1,0) \Psi_{f^{\star},R^{\star}}(0, t_2) + \textbf{Z}_n(0, t_2) \Psi_{f^{\star},R^{\star}}(t_1, 0) \right] \right.\\
 &&\left.-
 \textbf{Z}_n(t_1,t_2) \Psi_{f^{\star},R^{\star}}(t_1,0) \Psi_{f^{\star},R^{\star}}(0,t_2)
 \right\} dt_1 dt_2 + O_{\mathbb{P}}(\frac{1}{\sqrt{n}})
\end{eqnarray*}
where all $O_{\mathbb{P}}$ (and later $o_{\mathbb{P}}$) are in $\mathbb{P}_{C^{\star},R^{\star},f^{\star},\mathbb{Q}^{\star}}$ probability. Now, the empirical process converges uniformly in distribution to a Gaussian process over the set of functions $\{Id, \exp(it^T\cdot), |t|\leq \nu_{\text{est}}\}$, so that
 $\sqrt{n}\left(\frac{1}{n}\sum_{i=1}^{n}Y_{i}-E(Y_{1}),  M'_n(R^{\star})\right)$ converges in distribution to $\mathcal{N}(0,V)$ as $n$ goes to infinity for $V$ the covariance matrix of the random vector.

Let us now prove (3).
Twice differentiation of $M$ gives
\begin{eqnarray*}
M''(R) &= &\int_{B_{\nu}\times B_{\nu}^{d-1}} |\Phi_{\epsilon}(t_1,t_2)|^2 \left( \frac{d^2}{dR^2}  m_R(t_1,t_2) m_R(-t_1,-t_2) + 2 \frac{d}{dR} m_R(t_1,t_2)\right.\\
&&\left.\frac{d}{dR}m_R(-t_1,-t_2) + \frac{d^2}{dR^2} m_R(-t_1,-t_2) m_R(t_1,t_2) \right)  dt_1 dt_2.
\end{eqnarray*}
But $m_{R^{\star}}(t_1,t_2) = 0$ for all $(t_1,t_2)$, so that
$$ M''(R^{\star}) = 
 2 \int_{B_{\nu}\times B_{\nu}^{d-1}} \left \vert \frac{d}{dR} m_{R^{\star}}(t_1,t_2) \right \vert^2 |\Phi_{\epsilon}(t_1,t_2)|^2 dt_1 dt_2 .$$
We shall prove $M''(R^{\star}) \neq 0$ by contradiction.\\
If it is not the case, we have, for almost all $(t_1,t_2) \in B_{\nu}\times B_{\nu}^{d-1}$, $\frac{d}{dR} m_{R^{\star}}(t_1,t_2) \Phi_{\epsilon}(t_1,t_2) = 0$.
Now, there exists $r_{\epsilon} \in ( 0,\nu) $ such that for all $(t_{1},t_{2})\in B_{r_{\epsilon}}\times B_{r_{\epsilon}}^{d-1}$, $\Phi_{\epsilon}(t_1,t_2)\neq 0$. Since $\frac{d}{dR}m_{R^{\star}} $ is a continuous function on $\mathbb{C}^d$
 we get  
 $\frac{d}{dR} m_{R^{\star}}(t_1,t_2) = 0$ for all $(t_1,t_2) \in B_{r_{\epsilon}}\times B_{r_{\epsilon}}^{d-1}$, that is
 \begin{eqnarray} 
 \frac{d}{dR} \Psi_{f^{\star},R^{\star}}(t_1,t_2) \Psi_{f^{\star},R^{\star}}(t_1,0) \Psi_{f^{\star},R^{\star}}(0,t_2) 
 &=&\Psi_{f^{\star},R^{\star}}(t_1,t_2) \frac{d}{dR} \Psi_{f^{\star},R^{\star}}(t_1, 0) \Psi_{f^{\star},R^{\star}}(0,t_2) \nonumber\\
& +& \Psi_{f^{\star},R^{\star}}(t_1,t_2) \Psi_{f^{\star},R^{\star}}(t_1, 0) \frac{d}{dR} \Psi_{f^{\star},R^{\star}}(0,t_2),
\label{eq:dm}
\end{eqnarray}
 with
 $$ \Psi_{f^{\star},R^{\star}}(t_1,t_2) = \int_{(0,1)^{d-1}} f^{\star}(u) \exp(i R^{\star}  t^{\intercal} S(u) ) du $$
 and
 $$\frac{d}{dR} \Psi_{f^{\star},R^{\star}}(t_1,t_2) = i \int_{(0,1)^{d-1}} [t^{\intercal} S(u) ] f^{\star}(u) e^{i R^{\star} t^{\intercal} S(u)} du. $$
 But $\Psi_{f^{\star},R^{\star}}$ and $\frac{d}{dR} \Psi_{f^{\star},R^{\star}}$ are multivariate analytic functions, so that, using Lemma C.1 in \cite{LCGLSup2021}, we have that \eqref{eq:dm} holds for all $(t_1,t_2) \in \C \times \C^{d-1}$. We shall now investigate the set of zeros of the functions  $\Psi_{f^{\star},R^{\star}}(\cdot,0)$ and $\frac{d}{dR} \Psi_{f^{\star},R^{\star}}(\cdot, 0)$. Let $t_{1}\in\C$ be such that $\Psi_{f^{\star},R^{\star}}(t_1,0)=0$. Then by Lemma \ref{lem:condi} it is possible to choose $t_{2}\in \C^{d-1}$ such that $\Psi_{f^{\star},R^{\star}}(t_1,t_2)\neq 0$, and also such that 
 $\Psi_{f^{\star},R^{\star}}(0,t_2)\neq 0$ since $\Psi_{f^{\star},R^{\star}}(0,\cdot)$ is a multivariate analytic function having only isolated zeros. Equation \eqref{eq:dm} then leads to 
 $\frac{d}{dR} \Psi_{f^{\star},R^{\star}}(t_{1}, 0)=0$ so that the set of zeros of the function  $\Psi_{f^{\star},R^{\star}}(\cdot,0)$ is a subset of that of the function $\frac{d}{dR} \Psi_{f^{\star},R^{\star}}(\cdot, 0)$. 
 Then, using Hadamard's factorization theorem (see \cite{Stein:complex} Chapter 4, Theorem 4.1), and the fact that $\Psi_{f^{\star},R^{\star}}(\cdot,0)$ and $\frac{d}{dR} \Psi_{f^{\star},R^{\star}}(\cdot, 0)$ have exponential growth order $1$,
we get that there exists an entire function $G$ of exponential growth order $1$ such that for any $t_{1}\in\C$,
$$
\frac{d}{dR} \Psi_{f^{\star},R^{\star}}(t_{1}, 0)=\Psi_{f^{\star},R^{\star}}(t_{1},0)G(t_{1}).
$$
Plugging into \eqref{eq:dm} we get that for all $(t_1,t_2) \in \C \times \C^{d-1}$,
$$
\frac{d}{dR} \Psi_{f^{\star},R^{\star}}(t_1,t_2)  \Psi_{f^{\star},R^{\star}}(0,t_2) 
 =\Psi_{f^{\star},R^{\star}}(t_1,t_2) G(t_{1}) \Psi_{f^{\star},R^{\star}}(0,t_2) 
 + \Psi_{f^{\star},R^{\star}}(t_1,t_2)  \frac{d}{dR} \Psi_{f^{\star},R^{\star}}(0,t_2).
 $$
The same arguments applied for each coordinate of $t_2$ gives that there exists a multivariate anlytic function $H$ such that
for any $t_{2}\in\C^{d-1}$,
$$
\frac{d}{dR} \Psi_{f^{\star},R^{\star}}(0,t_{2})=\Psi_{f^{\star},R^{\star}}(0,t_{2})H(t_{2}),
$$
so that 
 for all $(t_1,t_2) \in \C \times \C^{d-1}$,
 \begin{equation}
 \label{eq:dm2}
\frac{d}{dR} \Psi_{f^{\star},R^{\star}}(t_1,t_2)   
 =\Psi_{f^{\star},R^{\star}}(t_1,t_2) \left(G(t_{1})
 +H(t_{2})\right). 
 \end{equation}
But for any $t\in \C^{d}$,
$$\frac{d}{dR} \Psi_{f^{\star},R^{\star}}(t)=\frac{1}{R}\frac{d}{du}\Psi_{f^{\star},R^{\star}}(ut),\;u\in\R
$$
so that solving the derivative equation \eqref{eq:dm2} we find that 
$\Psi_{f^{\star},R^{\star}}(t_1,t_2)$ is a product of a function of $t_1$ only by a function of $t_2$ only, meaning that
$S^{(1)} (U)$ and $S^{(2)} (U)$ are independent variables, which is not true and we get a contradiction.
We conclude that $M''(R^{\star}) \neq 0$. 

To end the proof of (3), for all $R \in [0,+\infty[$, 
\begin{eqnarray*}
 M_n''(R) - M_n''(R^{\star}) &= &\int_{B_{\nu_{\text{est}}}\times B_{\nu_{\text{est}}}^{d-1}} |\tilde{\psi}_n(t_1,t_2)|^2 [a_1(t_1,t_2,R) - a_1(t_1,t_2,R^{\star})] dt_1 dt_2 \\
&& + \int_{B_{\nu_{\text{est}}}\times B_{\nu_{\text{est}}}^{d-1}} |\tilde{\psi} _n(t_1,0)|^2 |\tilde{\psi}_n(0,t_2)|^2 [a_2(t_1,t_2,R) - a_2(t_1,t_2,R^{\star})] dt_1 dt_2\\
 &&+ \text{Re} \Bigg\lbrace \int_{B_{\nu_{\text{est}}}\times B_{\nu_{\text{est}}}^{d-1}} \tilde{\psi} _n(-t_1,t_2) \tilde{\psi} _n(t_1,0) \tilde{\psi} _n(0,t_2) [a_3(t_1,t_2,R) - a_3(t_1,t_2,R^{\star})] dt_1, dt_2 \Bigg\rbrace ,
\end{eqnarray*}
for functions $a_1$, $a_2$ and $a_3$ functions that are, for all $(t_1,t_2)$, continuous in the variable $R$ and uniformly upper bounded for bounded $R$. Since for all $(t_1,t_2)$, $|\tilde{\psi}_n(t_1,t_2)|\leq 1$, we get that $\left|M_n''(R) - M_n''(R^{\star})\right|$ is upper bounded by
$$
\int_{B_{\nu_{\text{est}}}\times B_{\nu_{\text{est}}}^{d-1}}
\left(\left|a_1(t_1,t_2,R) - a_1(t_1,t_2,R^{\star})\right|
+\left|a_2(t_1,t_2,R) - a_2(t_1,t_2,R^{\star})\right|+\left|a_3(t_1,t_2,R) - a_3(t_1,t_2,R^{\star})\right|
\right)dt_1, dt_2
$$
from which, applying the continuity theorem,  we deduce that $M_n''(R_n) - M_n''(R^{\star})$ converges in probability to $0$ whenever $R_n$ is a random variable converging in probability to $R^{\star}$. Then, for any random variable $R_n \in [R_{\text{min}},R_{\text{max}}]$ converging in probability to $R^{\star}$,  $M''_n(R_n)$ converges in probability to $M''(R^{\star})$.

\subsection{Proof of Theorem \ref{theo:semipara}}
\label{subsec:semipara}

Using Taylor expansion of $M'_n$ near $R^{\star}$, there exists $R_n \in (\widetilde{R},R^{\star})$ such that
$$ M'_n(\widetilde{R}) = M'_n(R^{\star}) + (R_n - R^{\star}) M''_n(R_n).$$
Using $M'_n(\widetilde{R}) = 0$ and Lemma \ref{lem:convgauss} we get $M''_n(R_n) =  M''(R^{\star}) + o_{\mathbb{P}}(1) $, so that 
$$ \sqrt{n} (\widetilde{R} - R^{\star}) = -\sqrt{n} \frac{M'_n(R^{\star})}{ M''(R^{\star})} (1 + o_{\mathbb{P}}(1)).$$
We deduce that
$$ \sqrt{n} \begin{pmatrix} \widetilde{R} - R^{\star} \\ \widetilde{C} - C^{\star} \end{pmatrix} = \begin{pmatrix} 0 \\ \sqrt{n}(\frac{1}{n} \sum_{l=1}^n Y_l - \mathbb{E}[Y_1]) \end{pmatrix} - \begin{pmatrix} 1 \\ - \mathbb{E}[ S(U) ] \end{pmatrix}  \frac{ \sqrt{n} M'_n(R^{\star})}{ M''(R^{\star})}(1 + o_{\mathbb{P}}(1))
$$
and the conclusion follows.

\section{Results on Bessel functions}
\label{sec:Bessel}

Denote $J_{\alpha}$ the Bessel function of order $\alpha \in [0, + \infty [$.\\
We shall use the following  results that can be found in \cite{Watson44}.

\begin{enumerate}

\item[(I)] The Bessel function of order $\alpha \in [0, + \infty [$ can be represented as

$$ J_{\alpha}(z) = \sum_{m = 0}^{\infty} (-1)^m \frac{(z/2)^{\alpha +2m}}{m! \Gamma(\alpha + m + 1)}, \ \ \ z \in \mathbb{C},$$

where for all $z \in ]0, + \infty[$, $\Gamma(z) = \int_{- \infty}^{+ \infty} t^{z-1} e^{-t} dt$. \\

\item[(II)] For $k \geq 0$ and $z \in \mathbb{C}$

$$J_{-k}(z) = (-1)^k J_k(z).$$

\item[(III)] For $z \in \mathbb{C}$ and $\theta  \in \mathbb{R}$

$$\text{exp}(izcos(\theta)) = \sum_{k \in \mathbb{Z}} i^k J_k(z) e^{-ik\theta}.$$

\item[(IV)] For $k \geq 0$ and $x,y > 0$

$$ |J_k(x) - J_k(y)| \leq |x-y|. $$

Indeed, since, $J_k \in \mathcal{C}^1(0,+\infty)$, for $k>0$, $J_k'(z) = \frac{1}{2}(J_{k-1}(z) - J_{k+1}(z))$, $J_0'(z) = -J_1(z)$ and $|J_k(x)| \leq 1$.

\item[(V)] For $\alpha \geq 1$ and $x >0$

$$ J_{\alpha+1}(x) = \frac{2 \alpha}{x}J_{\alpha}(x) - J_{\alpha -1}(x).$$

\end{enumerate}

We prove lemmas giving useful lower bounds.

\begin{lemma}\label{bessel1}

For all $0 \leq x < 1$, for all $\alpha \in [0, + \infty[$, we have

$$ J_{\alpha}(x) \geq \frac{x^{\alpha}}{2^{\alpha} \Gamma(\alpha + 1)} (1 - \frac{x^2}{4(\alpha+1)}) .$$

\end{lemma}

\begin{proof}

Let  $0 \leq x < 1$, for $\alpha \geq 0$, we have $J_{\alpha}(x) = \sum_{m = 0}^{\infty} (-1)^m \frac{(x/2)^{\alpha+2m}}{m! \Gamma(\alpha + m + 1)}$, so, if we expand the sum using that $\Gamma(x + 1) = x \Gamma(x)$ for $x > 0$ :

$$ J_{\alpha}(x) = \frac{x^{\alpha}}{2^{\alpha} \Gamma(\alpha + 1)}(1 - \frac{x^{2}}{4 (\alpha+1)} + \sum_{m \geq 2}^{\infty}  (-1)^m \frac{(x/2)^{2m}}{m!(\alpha+1) \cdots (\alpha+m)}).$$

Since $ 0 \leq x < 1$, we have $\sum_{m = 2}^{\infty}  (-1)^m \frac{(x/2)^{2m}}{m!(\alpha+1) \cdots (\alpha+m)} \geq 0$ thus ,

$$ J_k(x) \geq \frac{x^{\alpha}}{2^{\alpha} \Gamma(\alpha + 1)}(1 - \frac{x^{2}}{4 (\alpha+1)}).$$

\end{proof}

\begin{lemma}\label{lemmabessel}

For all $R>0$ and $0 < \nu < 1/R$, for all $0 \leq k \leq N$, we have :

$$ \int_{0}^{\nu} r J_k(rR)^2 dr \geq \frac{9 \nu^2}{32} \frac{(\nu R)^{2N}}{(N+1) 2^{2N} (N!)^2}.$$

\end{lemma}

\begin{proof}

Let  $R>0$ and $0 \leq r \leq \nu < 1/R$. For all $k \geq 0$, since $ 0 \leq rR < 1$, we have from lemma \ref{bessel1},
$$ J_k(rR) \geq \frac{(rR)^{k}}{2^k k!}(1 - \frac{(rR)^{2}}{4 (k+1)}),$$
and
$$ r J_k(rR)^2 \geq \frac{r(rR)^{2k}}{2^{2k} (k!)^2}(1 - \frac{(rR^{\star})^{2}}{4(k+1)})^2.$$
Then,
$$ \int_{0}^{\nu} r J_k(rR)^2 dr \geq \int_{0}^{\nu} \frac{r(rR)^{2k}}{2^{2k} (k!)^2}(1 - \frac{(rR)^{2}}{4 (k+1)})^2 dr \geq \frac{R^{2N}}{(2N+2) 2^{2N} (N!)^2}(1 - \frac{(\nu R)^2}{4})^2 \nu^{2N+2} .$$
To conclude the proof, we use that $ \nu < 1/R$, so that $(1-\frac{(\nu R)^2}{4})^2 \geq \frac{9}{16}$, which gives the result.

\end{proof}

\subsection*{Acknowledgements}

Jérémie Capitao-Miniconi would like to thank the IA Chair BisCottE (ANR-19-CHIA-0021-01), Elisabeth Gassiat would like to thank Institut Universitaire de France for supporting this project.

\bibliography{bib}
\end{document}